\documentclass[11pt]{amsart}
\usepackage{graphicx}

\usepackage[english]{babel}
\usepackage[latin1]{inputenc}
\usepackage{amsfonts}
\usepackage{amsthm}
\usepackage{amsmath}
\usepackage{amssymb}
\usepackage{amscd}
\usepackage{verbatim}
\usepackage{multicol}
\usepackage{tabularx}
\usepackage{amssymb}
\usepackage[all]{xy}
\usepackage[parfill]{parskip}

\theoremstyle{definition}
\newtheorem{thm}{Theorem}[section]
\newtheorem{pps}[thm]{Proposition}

\newtheorem{dfn}[thm]{Definition}

\DeclareMathOperator{\Tr}{\mathrm{Tr}}

\DeclareMathOperator{\Ric}{\mathrm{Ric}}

\DeclareMathOperator{\Ad}{\mathrm{Ad}}

\DeclareMathOperator{\GL}{\mathrm{GL}}

\DeclareMathOperator{\vspan}{\mathrm{span}}
\DeclareMathOperator{\diag}{\mathrm{diag}}

\def\ord#1^#2{#1$^{\text{#2}}$}

\def\lie#1{\mathfrak{#1}}

\def\hlie#1{\hat{\mathfrak{#1}}}

\def\uqr#1^#2{\text{$U_q^{#2}(\lie #1)$}}

\def\uqhr#1^#2{\text{$U_q^{#2}(\hlie #1)$}}
\def\us#1^#2{\text{$U_{\xi}^{#2}(\lie #1)$}}
\def\ush#1^#2{\text{$U_{\xi}^{#2}(\hlie #1)$}}
\def\dus#1^#2{\text{$\dot{U}_{\xi}^{#2}(\lie #1)$}}
\def\dush#1^#2{\text{$\dot{U}_{\xi}^{#2}(\hlie #1)$}}

\def\opl_#1^#2{\text{\scriptsize$\bigoplus\limits_{\text{\footnotesize$#1$}}^{\text{\footnotesize$#2$}}$}}
\def\otm_#1^#2{\text{\scriptsize$\bigotimes\limits_{\text{\footnotesize$#1$}}^{\text{\footnotesize$#2$}}$}}

\renewcommand{\thefootnote}

\allowdisplaybreaks
\begin{document}

\flushbottom

\title[]{
{ Invariant Einstein metrics on real flag manifolds with two or three isotropy summands}} 

\author[]{Brian Grajales and Lino Grama}

\address{IMECC-Unicamp, Departamento de Matem\'{a}tica. Rua S\'{e}rgio Buarque de Holanda,
651, Cidade Universit\'{a}ria Zeferino Vaz. 13083-859, Campinas - SP, Brazil.} 

\maketitle
\centerline{\small{
\begin{minipage}{350pt}
{\bf Abstract.} We study the existence of invariant Einstein metrics on real flag manifolds associated to simple and non-compact split real forms of complex classical Lie algebras whose isotropy representation decomposes into two or three irreducible subrepresentations. In this situation, one can have equivalent submodules, leading to the existence of non-diagonal homogeneous Riemannian metrics. In particular, we prove the existence of non-diagonal Einstein metrics on real flag manifolds.  
\end{minipage}}}
\section{Introduction}
A classical problem in differential geometry is the description of Einstein Riemannian metrics on a differentiable manifold $M$. A Riemannian metric $g$ is an {\em Einstein metric} if its associated Ricci tensor satisfies the condition $\Ric=cg,$ for some constant $c.$ This condition is usually called {\em Ricci equation} and the number $c$ is called {\em Einstein constant}. Those metrics are important to study several problems of Geometry and Physics (see \cite{Bes}). In particular, they appear as solutions of the Einstein field equations for the interaction between gravity and space-time in the vacuum. From a variational point of view, we can see Einstein metrics on a  differentiable manifold as critical points of the total {\em scalar curvature functional} restricted to the set of Riemannian metrics of volume 1 (see \cite{Bes}, \cite{BWZ}).

In general, Ricci equation becomes a system of partial differential equations. In the context of a homogeneous space provided with an invariant metric, this equation is equivalent to a system of algebraic equations. A complete classification of invariant Einstein metrics on homogeneous spaces is still an open problem. There are examples of homogeneous spaces admitting infinitely many invariant Einstein metrics up to homotheties (see for instance \cite{ADF}, \cite{WZ}), in this case, the isotropy representation has equivalent irreducible summands. In \cite{BWZ}; Böhm, Wang and Ziller conjectured that for a compact homogeneous space  whose isotropy representation decomposes into parwise inequivalent irreducible summands, the number of invariant Einstein metrics is finite. 

A important class of homogeneous space with parwise inequivalent irreducible isotropic summands is the so called {\em (complex) generalized flag manifolds} (or {\em Kähler C-spaces)}.  The problem of classification of Einstein metrics on complex flag manifolds has been widely studied (e.g. \cite{A}, \cite{AC}, \cite{AC1}, \cite{K}  and \cite{Y}). It is worth to point out that complex flag manifolds always admit at least one Einstein metric: the so called Einstein-Kähler metric (see \cite{Bes}).


A fundamental ingredient to study the geometry of homogeneous space is the description of the isotropy representation and its irreducible components. Using the isotropy representation one can, for instance, describe the space of invariant tensors. In the case of  { \em complex} flag manifolds the isotropy representation was described by de Siebenthal in \cite{sibe} (see also \cite{alek}). The description of the isotropy representation of {\em real} flag manifolds was work out recently by Patrão and San Martin in \cite{PSM}. One of the main difference between {\em real} and {\em complex} flag manifolds is the following: there exist real flag manifolds whose the isotropy representation decompose into irreducible and {\em equivalent} summands. This never happens in a complex flag manifold. This suggest that the invariant geometry of real flag manifolds is very rich, and some new phenomena can occur in the real case.  For instance, we shall see we can find {\em non-diagonal} invariant Einstein metrics (see Sections \ref{ss4.2.1} and \ref{ss4.2.6}). In the {\em complex} case every invariant Einstein metric is {\em diagonal}. For more recent developments about the geometry of real flag manifolds we suggest \cite{FBSM}, \cite{GGN}.

In this paper we deal with the problem of classification of invariant Einstein metrics on real flag manifolds. More precisely, we study the Ricci equation for {\em generalized real flag manifolds}  associated to a split real form $\mathfrak{g}$ of a complex simple Lie algebra of classical type whose isotropy representation decomposes into two or three irreducible summands. The classification of invariant Einstein metrics on {\em complex} flag manifolds was carried out by Arvanitoyeorgos and Chrysikos (\cite{AC1}) in the case of two isotropy summands and by Kimura (\cite{K}) in the case of three isotropy summands. We will deal with Einstein metric on real flag manifold associated to exceptional Lie groups in a forthcoming paper.

The following table summarize our results about the number of invariant Einstein metrics in real flag manifolds. The last column indicates when the {\em normal metric} (induced by the Cartan-Killing form) is either Einstein or not. 
\begin{center}
	$\begin{array}{|c|c|c|c|c|}
	\hline \scriptsize\mathbb{F}_\Theta & \text{\scriptsize \# Isotropy}& \text{\scriptsize Equivalent} & \text{\scriptsize \# Einstein}&\text{\scriptsize Normal}\\
	&\text{\scriptsize summands}&\scriptsize\text{summands?}&\scriptsize\text{metrics}&\\\hline
	\scriptstyle\frac{SO(4)}{S(O(2)\times O(2))}& \scriptstyle2 & - & \scriptstyle1 & \scriptsize\checkmark \\\hline
	\scriptstyle \frac{SO(l)\times SO(l+1)}{SO(l-1)\times SO(l)},\ l\geq 3& \scriptstyle2 & - & \scriptstyle1 & - \\\hline
	\scriptstyle \frac{SO(l)\times SO(l+1)}{SO(l)},\ l\geq 3,\ l\neq 4 & \scriptstyle2 & - & \scriptstyle2 & - \\\hline
	\scriptstyle \frac{U(l)}{O(l)},\ l\geq 3& \scriptstyle2 & - & \scriptstyle0 & - \\\hline
	\scriptstyle \frac{U(l)}{O(1)\times U(l-1)},\ l\geq 3& \scriptstyle2 & - & \scriptstyle1 & - \\\hline
	\scriptstyle \frac{SO(4)\times SO(4)}{SO(4)}& \scriptstyle2 & - & \scriptstyle1 & \scriptsize\checkmark \\\hline
	\scriptstyle\frac{SO(4)}{S(O(2)\times O(1)\times O(1))}& \scriptstyle3 & \scriptsize\checkmark & \scriptstyle5 & - \\\hline
	\scriptstyle\frac{SO(l+1)}{S(O(l_1)\times O(l_2)\times O(l_3))},\ l\geq2,\ l\neq 3,\ l_1+l_2+l_3=l+1 & \scriptstyle3 & - & \scriptstyle\leq 4 & -  \\\hline
	\scriptstyle\frac{SO(4)\times SO(5)}{SO(4)} & \scriptstyle3 & - & \scriptstyle2 & \scriptsize\checkmark \\\hline
	\scriptstyle\frac{SO(l)\times SO(l+1)}{SO(d)\times SO(l-d)\times SO(l-d+1)},\ l\geq 3,\ 2\leq d\leq l-1 & \scriptstyle3 & - & \scriptstyle\leq 4 & -  \\\hline
	\scriptstyle\frac{U(l)}{O(d)\times U(l-d)},\ l\geq 3,\ 2\leq d\leq l-1 & \scriptstyle3 & - & \scriptstyle\leq 2 & -  \\\hline
	\scriptstyle\frac{SO(l)\times SO(l)}{S(O(l-1)\times O(1))},\ l\geq 4 & \scriptstyle3 & \scriptsize\checkmark & \scriptstyle6\ \text{(5 when}\ l=4) & \scriptstyle\text{only for}\ l=4 \\\hline
	\end{array}$
	
	\small{Table 1. Real flag manifolds with two or three isotropy summands}
\end{center}
It is worth to point out that the flag manifold  $SO(l+1)/S(O(l_1)\times O(l_2)\times O(l_3)),\ l\geq2,\ l\neq 3,\ l_1+l_2+l_3=l+1$ admit at most four invariant Einstein metrics and there exist families of such flag manifolds admitting two, three or four Einstein metrics depending on the choice of $ l_1,l_2,l_3$. We will exhibit (family of) examples where each situation occur. Similar remarks hold for the manifolds $SO(l)\times SO(l+1)/SO(d)\times SO(l-d)\times SO(l-d+1),\ l\geq 3,\ 2\leq d\leq l-1$; and $U(l)/O(d)\times U(l-d), l\geq 3,\ 2\leq d\leq l-1$.

Our paper is organized as follows: in Section 2 we quickly review the description of the Ricci tensor of an invariant metric. In Section 3, after a review of the description of the isotropy representation of a real flag manifolds we {\em classify} real flag manifolds of classical Lie groups with two or three isotropy summands. In Section 4 we describe explicitly the invariant Einstein metric in each real flag manifold listed in the previous section. Finally in Section 5 we analyze the problem of isometry of the invariant Einstein metrics on real flag manifolds.

\section{The Ricci tensor}\label{section2}
Consider a homogeneous space $M=G/H,$ where $G$ is a compact, connected Lie group and $H$ is a closed subgroup of $G.$ A Riemannian metric $(\cdot,\cdot)$ on $G/H$ is said to be $G-$\textit{invariant} if, for every $a\in G,$ the map
\begin{equation}
\phi_a:(G/H,(\cdot,\cdot))\longrightarrow (G/H,(\cdot,\cdot));\ \ \phi_a(bH)=abH
\end{equation}
is an isometry. Let $\mathfrak{g}$ and $\mathfrak{h}$ be the Lie algebras of $G$ and $H$ respectively, by compactness of $G$, there exists a reductive decomposition $\mathfrak{g}=\mathfrak{h}\oplus\mathfrak{m},$ i.e., $\Ad(h)\mathfrak{m}\subseteq\mathfrak{m}$. The subspace $\mathfrak{m}$ is identified with the tangent space of $G/H$ at the identity coset  $eH$ via the isomorphism
\begin{equation}
X\longmapsto X^*(eH)=\left.\frac{d}{dt}exp(tX)H\right|_{t=0}.
\end{equation} 
This map also gives rise to an equivalence between the isotropy representation of $G/H$ and the adjoint representation of $H$ on $\mathfrak{m}$. Thus, every $G-$invariant metric on $G/H$ can be identified with an $\Ad(H)-$invariant inner product $g$ on $\mathfrak{m}$, i.e.,
\begin{equation}
g(\Ad(h)X,\Ad(h)Y)=g(X,Y) \text{ for all } h\in H,\ X,Y\in\mathfrak{m}.
\end{equation}
Let $(\cdot,\cdot)$ be a fixed $\Ad(G)-$invariant inner product $(\cdot,\cdot)$ on $\mathfrak{g}$ such that $\mathfrak{g}=\mathfrak{h}\oplus\mathfrak{m}$ is an $(\cdot,\cdot)-$orthogonal reductive decomposition. Then, every $\Ad(H)-$invariant inner product $g$ on $\mathfrak{m}$ is completely determined by a unique $(\cdot,\cdot)-$self-adjoint, positive operator $A:\mathfrak{m}\longrightarrow\mathfrak{m}$ that commutes with $\Ad(h)|_{\mathfrak{m}},$ for all $h\in H.$ The operator $A$ is defined implicitly by the formula
\begin{equation}\label{n1}
g(X,Y)=(AX,Y) \text{ for all } X,Y\in\mathfrak{m}.
\end{equation}
Conversely, given an operator $A$ satisfying the conditions above, we have that the formula \eqref{n1} defines an $\Ad(H)-$invariant inner product on $\mathfrak{m}.$ We call $A$ the \textit{metric operator} corresponding to $g.$ From now, we identify a $G-$invariant Riemannian metric on $G/H$ with its corresponding $\Ad(H)-$invariant product on $M$ and its corresponding metric operator $A.$ 

By compactness of $H$ (as a closed subgroup of the compact group $G$), the adjoint representation of $H$ on $\mathfrak{m}$ induces a $(\cdot,\cdot)-$orthogonal splitting
\begin{equation}\label{3}
\mathfrak{m}=\bigoplus\limits_{i=1}^s\mathfrak{m}_i
\end{equation}
of $\mathfrak{m}$ into $H-$invariant, irreducible subspaces $\mathfrak{m}_i$, $i=1,...,s.$ When all the submodules $\mathfrak{m}_i$ have multiplicity one, every $G-$invariant metric $A$ is equal to a positive scalar multiple of the identity map when it is restricted to each $\mathfrak{m}_i.$ When $\mathfrak{m}_i$ and $\mathfrak{m}_j$ are equivalent for some $i,j$, we have metric operators $A$ mapping vectors of $\mathfrak{m}_i$ to vectors with non-zero projection on $\mathfrak{m}_j.$ 

For a $G-$invariant metric $g$ on $G/H$, define $U:\mathfrak{m}\times\mathfrak{m}\rightarrow\mathfrak{m}$ by the formula
\begin{equation}\label{U}
\displaystyle 2g(U(X,Y),W)=g([W,X]_{\mathfrak{m}},Y)+g([W,Y]_{\mathfrak{m}},X)
\end{equation} 
for all $W\in\mathfrak{m}$. We may apply Corollary 7.38 of \cite{Bes} and obtain an explicit formula for the Ricci tensor:
\begin{equation}\label{ric}
\begin{array}{ccl}
\Ric(X,Y) & = & \displaystyle-\frac{1}{2}\sum\limits_ig([X,X_i]_{\mathfrak{m}},[Y,X_i]_{\mathfrak{m}})-\frac{1}{2}\langle X,Y\rangle\\
 & & \\
 & & \displaystyle +\frac{1}{4}\sum\limits_{i,j}g([X_i,X_j]_{\mathfrak{m}},X)g([X_i,X_j]_{\mathfrak{m}},Y)-g(U(X,Y),Z)
\end{array}
\end{equation}
where $\langle\cdot,\cdot\rangle$ is the Killing form of $G$,  $\{X_i\}$ is an $g-$orthonormal basis of $\mathfrak{m}$, $Z=\displaystyle\sum\limits_i U(X_i,X_i)$ and $X,\ Y\in\mathfrak{m}.$ We say that $g$ is an Einstein metric if there exists a real number $c$ such that
\begin{equation}\label{Einstein}
\Ric=cg.
\end{equation}
We shall study equation \eqref{Einstein} for split real flag manifolds of classical Lie groups whose isotropy representation decomposes into two or three irreducible submodules. For a  non-diagonal invariant metric $g,$ one can have $g-$orthogonal vectors $X,Y\in\mathfrak{m}$ with $\Ric(X,Y)$ not necessarily zero. In this situation, $\Ric(X,Y)$ gives us an expression in terms of the parameters of the metric $g$ which can be factored. If $g$ satisfies equation \eqref{Einstein}, then $\Ric(X,Y)=0$ and we obtain necessary conditions for $g$ to be an Einstein metric.

\section{The isotropy representation of a real flag manifold}
Let $\mathfrak{g}$ be a non-compact, simple, real Lie algebra which is a split real form of a complex Lie algebra, $G$ a connected Lie group with Lie algebra $\mathfrak{g}$ and $P_\Theta$ a parabolic subgroup of $G.$ A \textit{generalized flag manifold} of $\mathfrak{g}$ is the quotient space $\mathbb{F}_\Theta=G/P_\Theta.$ If $K$ is a maximal compact subgroup of $G,$ then $K$ acts transitively on $\mathbb{F}_\Theta$ with isotropy $K_\Theta=K\cap P_\Theta,$ so $\mathbb{F}_\Theta$ can be identified with the quotient $K/K_\Theta$ as well. We fix an $\Ad(K)-$invariant inner product $(\cdot,\cdot)$ on the Lie algebra $\mathfrak{k}$ of $K$ and consider the reductive decomposition $\mathfrak{k}=\mathfrak{k}_\Theta\oplus\mathfrak{m}_\Theta$, where $\mathfrak{k}_\Theta$ is the Lie algebra of $K_\Theta$ and  $\mathfrak{m}_\Theta=\mathfrak{k}_\Theta^{\perp},$ so that the tangent space at $o=eK_\Theta$ can be identified with $\mathfrak{m}_\Theta.$ 

We can also describe generalized flag manifolds of $\mathfrak{g}$ by considering a Cartan decomposition $\mathfrak{g}=\mathfrak{k}\oplus\mathfrak{s}$ and a maximal abelian subalgebra $\mathfrak{a}\subseteq \mathfrak{s}$ as follows: let $\Pi$ denote the set of roots of $\mathfrak{g}$ corresponding to $\mathfrak{a}$ and 
\begin{center}
	$\mathfrak{g}=\displaystyle \mathfrak{g}_0\oplus\bigoplus\limits_{\alpha\in\Pi}\mathfrak{g}_\alpha$
\end{center}
the associated root space decomposition. Fix a set $\Pi^+$ of positive roots and consider $\Sigma$ the corresponding set of simple roots. Each $\Theta\subseteq \Sigma$ determines a parabolic subalgebra 
\begin{center}
	$\mathfrak{p}_\Theta=\displaystyle\mathfrak{g}_0\oplus\bigoplus\limits_{\alpha\in\Pi^+}\mathfrak{g}_{\alpha}\oplus\bigoplus\limits_{\alpha\in\langle\Theta\rangle^-}\mathfrak{g}_\alpha$
\end{center}
where $\langle\Theta\rangle^-$ is the set of negative roots generated by $\Theta$. If $P_\Theta=\{a\in G\ |\ \Ad(a)\mathfrak{p}_\Theta\subseteq\mathfrak{p}_\Theta\},$ then $P_\Theta$ is a parabolic subgroup of $G$ and any parabolic subgroup of $G$ can be obtained in this form.

If $H_\alpha,\ \alpha\in\Sigma,\ X_\alpha\in\mathfrak{g}_\alpha,\ \alpha\in\Pi$ is a Weyl basis of the complexified Lie algebra $\mathfrak{g}_{\mathbb{C}}$ of $\mathfrak{g}$, we denote by $\mathfrak{m}_\alpha=\mathfrak{m}_{-\alpha}=\vspan\{X_\alpha-X_{-\alpha}\},$ so we have
\begin{equation*}
\mathfrak{k}=\bigoplus\limits_{\alpha\in\Pi^-}\mathfrak{m}_\alpha,\ \mathfrak{k}_\Theta=\bigoplus\limits_{\alpha\in\langle\Theta\rangle^-}\mathfrak{m}_\alpha\ \text{and}\ \mathfrak{m}_\Theta=\bigoplus\limits_{\alpha\in\Pi^{-}\setminus\langle\Theta\rangle^{-}}\mathfrak{m}_\alpha.
\end{equation*}
By compactness of $K_\Theta,$ the ajoint representation $\Ad:K_\Theta\longrightarrow \GL(\mathfrak{m}_\Theta)$ admits irreducible $(\cdot,\cdot)-$ortho- gonal subrepresentations $W_1,...,W_s$ such that 
\begin{equation}
\mathfrak{m}_\Theta=W_1\oplus...\oplus W_s.
\end{equation}
We shall present the description given by Patrão and San Martin in \cite{PSM} of these subrepresentations and the equivalences between them. 

\subsection{Flags of $A_l=\mathfrak{sl}(l+1,\mathbb{R})$} The special linear Lie algebra $\mathfrak{sl}(l+1,\mathbb{R})$ is composed by the real $(l+1)\times(l+1)$ matrices with trace zero. In this case, $\mathfrak{a}$ is the subalgebra of traceless diagonal matrices. The roots are given by $\alpha_{ij}=\lambda_i-\lambda_j,\ 1\leq i\neq j\leq l+1$, where
\begin{equation*}
\lambda_i:\mathfrak{a}\longrightarrow \mathbb{R},\ \lambda_i(\diag(a_1,...,a_{l+1}))=a_i,\ i=1,...,l+1.
\end{equation*}
The simple roots are $\alpha_i=\alpha_{i,i+1},\ i=1,...,l$.
The subalgebra $\mathfrak{k}=\mathfrak{so}(l+1)$ is the set of $(l+1)\times(l+1)$ skew-symmetric real matrices and the compact maximal subgroup $K=SO(l+1)$ is the set of orthogonal matrices of order $l+1$ with determinant 1. We consider the  $\Ad(SO(l+1))-$invariant inner product $(\cdot,\cdot)$ given by the negative of the Killing form of $\mathfrak{so}(l+1)$. For every root $\alpha_{ij}$ we have that $\mathfrak{m}_{\alpha_{ij}}=\vspan\{w_{ij}=E_{ij}-E_{ji}\}$, where $E_{ij}$ is the $(l+1)\times(l+1)$ matrix with value $1$ in the $(i,j)-$entry and $0$ elsewhere. Every set $\Theta\subseteq \Sigma$ is determined by positive numbers $l_1,...,l_r$ such that $l_1+...+l_r=l+1$ and 
\begin{equation}
\Theta=\bigcup\limits_{l_i>1}\{\alpha_{\tilde{l}_{i-1}+1},...,\alpha_{\tilde{l}_{i}-1}\},
\end{equation}
where $\tilde{l}_0=0$ and $\tilde{l}_i=\tilde{l}_{i-1}+l_i,\ i=1,...,r.$
\begin{pps}\label{IsotropyA}(\cite{PSM}) Let $\mathbb{F}_\Theta$ be a flag of $A_l=\mathfrak{sl}(l+1,\mathbb{R}).$ Then the subspaces 
	\begin{equation}\label{SummandsA}
	M_{mn}=\bigoplus\limits_{\begin{subarray}{c}
		\tilde{l}_{m-1}<i\leq\tilde{l}_m\\
		\tilde{l}_{n-1}<j\leq\tilde{l}_n
		\end{subarray}}\mathfrak{m}_{\alpha_{ij}},\ 1\leq n<m\leq r,
	\end{equation}
	are $(\cdot,\cdot)-$orthogonal submodules of the adjoint representation of $K_\Theta$ in $\mathfrak{m}_\Theta.$ Moreover we have:
	
	$a)$ If $l\neq 3$ or $l=3$ and $\Theta=\{\alpha_1,\alpha_2\}$ or $\{\alpha_2,\alpha_3\}$, then the submodules $M_{mn}$ are irreducible and pairwise inequivalent.
	
	$b)$ If $l=3$ and $\Theta=\emptyset$, then the submodules $M_{mn}$ are irreducible and $M_{mn}$ is equivalent to $M_{m'n'}$ if and only if $\{m,n,m',n'\}=\{1,2,3,4\}$.
	
	$c)$ If $l=3$ and $\Theta=\{\alpha_i\}$ for some $i\in\{1,2,3\}$ then the submodules $M_{mn}$ are irreducible and $M_{mn}$ is equivalent to $M_{m'n'}$ if and only if $i\in\{m,n\}\cap\{m',n'\}.$
	
	$d)$ If $l=3$ and $\Theta=\{\alpha_1,\alpha_3\}$ then $M_{21}$ decomposes into the inequivalent $K_\Theta-$irreducible subspaces 
	\begin{equation*}
		M_1=\vspan\{w_{31}-w_{42},w_{41}+w_{
			32}\}\ \text{and}\ M_2=\vspan\{w_{31}+w_{42},w_{41}-w_{32}\}.
	\end{equation*}\hfill $\qed$
\end{pps}
\subsection{Flags of $B_l=\mathfrak{sl}(l+1,l)$} The Lie algebra $B_l$ is the set
\begin{equation*}
\mathfrak{sl}(l+1,l)=\left\{\left(\begin{array}{ccc}
0 & -a & -b \\
b^T & A & B \\
a^T & C & -A^T
\end{array}\right)\in\mathfrak{gl}(2l+1,\mathbb{R})\ :\ B+B^T=C+C^T=\textbf{0}\right\}.
\end{equation*}
We consider the abelian subalgebra
\begin{equation*}
\mathfrak{a}=\left\{\left(\begin{array}{ccc}
0 & \textbf{0} & \textbf{0} \\
\textbf{0} & \Lambda & \textbf{0} \\
\textbf{0} & \textbf{0} & -\Lambda
\end{array}\right)\in\mathfrak{sl}(l+1,l)\ :\ \Lambda=\diag(a_1,...,a_l)\right\}.
\end{equation*}
The roots are given by $\pm(\lambda_i-\lambda_j),\ \pm(\lambda_i+\lambda_j),\ 1\leq i<j\leq l$ and $\lambda_i,\ 1\leq i\leq l$, where 
\begin{equation*}
\lambda_i:\mathfrak{a}\longrightarrow \mathbb{R};\ \lambda_i\left(\left(\begin{array}{ccc}
0 & \textbf{0} & \textbf{0} \\
\textbf{0} & \Lambda & \textbf{0} \\
\textbf{0} & \textbf{0} & -\Lambda
\end{array}\right)\right)=a_i,\ \Lambda=\diag(a_1,...,a_l),\ i=1,...,l.
\end{equation*}
The set of simple roots is $\Sigma=\{\lambda_i-\lambda_{i+1}:1\leq i\leq l-1\}\cup\{\lambda_l\}.$ We denote by $\alpha_i=\lambda_i-\lambda_{i+1},\ 1\leq i\leq l-1$ and $\alpha_l=\lambda_l.$ The subalgebra $\mathfrak{k}$ is the set of skew-symmetric matrices in $\mathfrak{sl}(l+1,l)$, i.e.,
\begin{equation*}
\mathfrak{k}=\left\{\left(\begin{array}{ccc}
0 & -a & -a \\
a^T & A & B \\
a^T & B & A
\end{array}\right)\in\mathfrak{sl}(l+1,l)\ :\ A+A^T=B+B^T=\textbf{0}\right\}.
\end{equation*}
Fix the $\Ad(K)-$invariant inner product $(\cdot,\cdot)$ in $\mathfrak{k}$ given by 
\begin{equation}\label{ProductB}
\left(\left(\begin{array}{ccc}
0 & -a & -a \\
a^T & A & B \\
a^T & B & A
\end{array}\right),\left(\begin{array}{ccc}
0 & -c & -c \\
c^T & C & D \\
c^T & D & C
\end{array}\right)\right)=\displaystyle ac^T-\frac{\Tr(BD)+\Tr(AC)}{2}
\end{equation}
and define the matrices
\begin{equation}\label{BasisB}
\begin{array}{ll}
v_{k}=E_{1+k,1}-E_{1,1+k}+E_{1+l+k,1}-E_{1,1+l+k}, & 1\leq k \leq l, \\
w_{ij}=E_{1+i,1+j}-E_{1+j,1+i}+E_{1+l+i,1+l+j}-E_{1+l+j,1+l+i}, & \\
u_{ij}=E_{1+l+i,1+j}-E_{1+l+j,1+i}+E_{1+i,1+l+j}-E_{1+j,1+l+i}, & 1\leq j<i\leq l,\\
\end{array}
\end{equation}
where $E_{ij}$ is the $(2l+1)\times(2l+1)$-matrix with value equal to 1 in the $(i,j)-$entry and zero elsewhere. Then we have
\begin{equation*}
\mathfrak{m}_{\lambda_i-\lambda_j}=\vspan\{w_{ij}\},\ \mathfrak{m}_{\lambda_i+\lambda_j}=\vspan\{u_{ij}\},\ 1\leq j<i\leq l\ \text{and}\ \mathfrak{m}_{\lambda_k}=\vspan\{v_k\},\ 1\leq k\leq l. 
\end{equation*}
For every $\Theta\subseteq\Sigma,$ take $l_1,...,l_r$ such that $l_1+...+l_r=l$ and
\begin{equation*}
\Theta=\bigcup\limits_{l_i>1}\{\alpha_{\tilde{l}_{i-1}+1},...,\alpha_{\tilde{l}_i-1}\} \text{ or } \bigcup\limits_{l_i>1}\{\alpha_{\tilde{l}_{i-1}+1},...,\alpha_{\tilde{l}_i-1}\}\cup\{\alpha_l\},
\end{equation*} 
where $\tilde{l}_0=0$ and $\tilde{l}_i=\tilde{l}_{i-1}+l_i$, $i=1,...,r.$
\begin{pps}\label{IsotropyB}(\cite{PSM}) Let $\mathbb{F}_\Theta$ be a flag manifold of $B_l,$ with $l\geq 5.$ Then the following subspaces are $K_\Theta-$invariant and irreducible:
	
	$a)$\begin{center}
		$V_i=\bigoplus\limits_{\tilde{l}_{i-1}<k\leq\tilde{l}_i}\mathfrak{m}_{\lambda_k},$ $i=1,...,r$ 
	\end{center}
	when $\alpha_l\notin\Theta.$ All these subspaces are pairwise inequivalent.\\
	
	$b)$\begin{center}
		$W_{mn}=\displaystyle\bigoplus\limits_{\begin{subarray}{c} \tilde{l}_{m-1}< i\leq \tilde{l}_m\\
			\tilde{l}_{n-1}< j\leq \tilde{l}_n\end{subarray}}\mathfrak{m}_{\lambda_i-\lambda_j}$ and $U_{mn}=\displaystyle\bigoplus\limits_{\begin{subarray}{c} \tilde{l}_
			{m-1}< i\leq \tilde{l}_m\\
			\tilde{l}_{n-1}< j\leq \tilde{l}_n\end{subarray}}\mathfrak{m}_{\lambda_i+\lambda_j},$
	\end{center}
	with $1\leq n<m\leq r$ if $\alpha_l\not\in\Theta$ and $1\leq n<m\leq r-1$ if $\alpha_l\in\Theta.$ The subspace $W_{mn}$ is equivalent to $U_{mn}$ for each $(m,n)$. We denote by $M_{mn}=W_{mn}\oplus U_{mn}.$ \\
	
	$c)$\begin{center}
		$U_i=\displaystyle\bigoplus\limits_{\tilde{l}_{i-1}<t<s\leq\tilde{l}_i}\mathfrak{m}_{\lambda_s+\lambda_t}$
	\end{center}
	for $i$ such that $l_i>1$ and $1\leq i\leq r$ if $\alpha_l\notin\Theta$ and $1\leq i\leq r-1$ if $\alpha_l\in\Theta.$ These subspaces are pairwise inequivalent.\\
	
	$d)$\begin{center}
		$\begin{array}{ccl}
		(V_i)_1 & = & \vspan\{w_{\tilde{l}_{r-1}+s,\tilde{l}_{i-1}+t}-u_{\tilde{l}_{r-1}+s,\tilde{l}_{i-1}+t}:1\leq s \leq l_r, 1\leq t\leq l_i\}\\
		\\
		(V_i)_2 & = & \vspan\{v_{\tilde{l}_{i-1}+t},\ w_{\tilde{l}_{r-1}+s,\tilde{l}_{i-1}+t}+u_{\tilde{l}_{r-1}+s,\tilde{l}_{i-1}+t}:1\leq s \leq l_r, 1\leq t\leq l_i\}\\
		\end{array}$
	\end{center}
	with $1\leq i\leq r-1$ when $\alpha_l\in\Theta.$ These subspaces are pairwise inequivalent. \hfill $\qed$
\end{pps}
For $B_2,$ $B_3,$ $B_4$ and some subsets $\Theta\subseteq\Sigma$, one can have different equivalences among the $K_\Theta-$invariant irreducible subspaces above. 
Even more, there are $K_\Theta-$invariant subspaces different from those in the proposition above. As an example, consider the flag of $B_4$ given by $\Theta=\{\alpha_1,\alpha_2,\alpha_3\}$. In this case, $\mathfrak{m}_\Theta$ decomposes into  the $K_\Theta-$invariant, irreducible and pairqise inequivalent subspaces $V_1,\ T_1,\ T_2$, where $V_1$ is defined as before and
\begin{equation*}
T_1=\vspan\{u_{21}+u_{43},u_{31}-u_{42},u_{41}+u_{32}\}, T_2=\vspan\{u_{21}-u_{43},u_{31}+u_{42},u_{41}-u_{32}\}.
\end{equation*}
\subsection{Flags of $C_l=\mathfrak{sp}(l,\mathbb{R})$} The symplectic real Lie algebra $\mathfrak{sp}(l,\mathbb{R})$ is the set 
\begin{equation*}
\left\{\left(\begin{array}{cc}
A & B \\
C & -A^T\end{array}\right)\in\mathfrak{gl}(2l,\mathbb{R})\ : \ B-B^T=C-C^T=\textbf{0}\right\}
\end{equation*}
and the subalgebra $\mathfrak{a}$ is given by
\begin{equation*}
\mathfrak{a}=\left\{\left(\begin{array}{cc}
\Lambda & \textbf{0} \\
\textbf{0} & -\Lambda\end{array}\right)\in\mathfrak{sp}(l,\mathbb{R})\ : \ \Lambda=\diag(a_1,...,a_l)\right\}.
\end{equation*}
The roots are $\pm(\lambda_i-\lambda_j),\ \pm(\lambda_i+\lambda_j),\ 1\leq i<j\leq l$ and $2\lambda_i,\ 1\leq i\leq l,$ where
\begin{equation*}
\lambda_i:\mathfrak{a}\longrightarrow \mathbb{R};\ \lambda_i\left(\left(\begin{array}{cc}
\Lambda & \textbf{0} \\
\textbf{0} & -\Lambda
\end{array}\right)\right)=a_i,\ i=1,...,l.
\end{equation*}
We denote by $\alpha_i=\lambda_i-\lambda_{i+1},\ i=1,...,l-1$ and $\alpha_l=2\lambda_l$, so that the set of simple roots is given by $\Sigma=\{\alpha_1,...,\alpha_l\}$. The subalgebra $\mathfrak{k}$ is composed by the real square $2l\times 2l$ matrices of the form
\begin{equation*}
\left(\begin{array}{cc}
A & -B^T \\
B & A\end{array}\right),\ A+A^T=B-B^T=\textbf{0}.
\end{equation*}
In this case, we fix the $\Ad(K)-$invariant inner product $(\cdot,\cdot)$ in $\mathfrak{k}$ defined by 
\begin{equation}\label{ProductC}
\left(\left(\begin{array}{cc}
A & -B \\
B & A
\end{array}\right),\left(\begin{array}{cc}
C & -D \\
D & C
\end{array}\right)\right)=\displaystyle\frac{\Tr(BD)-\Tr(AC)}{2}.
\end{equation}
It is easy to see that the matrices 
\begin{equation}\label{BasisC}
\begin{array}{ll}
u_{kk}=E_{l+k,k}-E_{k,l+k}, & 1\leq k \leq l, \\
w_{ij}=E_{ij}-E_{ji}+E_{l+i,l+j}-E_{l+j,l+i}, & \\
u_{ij}=E_{l+i,j}+E_{l+j,i}-E_{i,l+j}-E_{j,l+i}, & 1\leq j<i\leq l,\\
\end{array}
\end{equation}
where $E_{ij}$ is the $2l\times 2l$ matrix with value equal to 1 in the $(i,j)-$entry and zero elsewhere, form an $(\cdot,\cdot)-$orthonormal basis of $\mathfrak{k}.$ Also we have
\begin{equation*}
\mathfrak{m}_{\lambda_i-\lambda_j}=\vspan\{w_{ij}\},\ \mathfrak{m}_{\lambda_i+\lambda_j}=\vspan\{u_{ij}\},\ 1\leq i<j\leq l\ \text{and}\ \mathfrak{m}_{2\lambda_k}=\vspan\{u_{kk}\},\ 1\leq k \leq l. 
\end{equation*} 
As before, let $l_1,...,l_r$ be positive numbers such that $l_1+...+l_r=l$ and
\begin{equation*}
\Theta=\bigcup\limits_{l_i>1}\{\alpha_{\tilde{l}_{i-1}+1},...,\alpha_{\tilde{l}_i-1}\} \text{ or } \bigcup\limits_{l_i>1}\{\alpha_{\tilde{l}_{i-1}+1},...,\alpha_{\tilde{l}_i-1}\}\cup\{\alpha_l\},
\end{equation*} 
where $\tilde{l}_0=0$ and $\tilde{l}_i=\tilde{l}_{i-1}+l_i$, $i=1,...,r.$
\begin{pps}\label{IsotropyC}(\cite{PSM}) Let $\mathbb{F}_\Theta$ be a flag manifold of $C_l,$ with $l\neq4.$ Then the following subspaces are $K_\Theta-$invariant and irreducible:\\
	
	$a)$ 
	\begin{center}$V_i=\vspan\{u_{\tilde{l}_{i-1}+1,\tilde{l}_{i-1}+1}+...+u_{\tilde{l}_i,\tilde{l}_i}\},$\end{center}
	with $1\leq i\leq r$ if $\alpha_l\not\in\Theta$ and $1\leq i\leq r-1$ if $\alpha_l\in\Theta.$ All these subspaces arepairwise equivalent. Set $M_0=V_1\oplus...\oplus V_{\tilde{r}}$, where $\tilde{r}=r$ if $\alpha_l\notin\Theta$ and $\tilde{r}=r-1$ if $\alpha_l\in\Theta.$\\
	
	$b)$
	\begin{center}
		$W_{mn}=\displaystyle\bigoplus\limits_{\begin{subarray}{c} \tilde{l}_{m-1}< i\leq \tilde{l}_m\\
			\tilde{l}_{n-1}< j\leq \tilde{l}_n\end{subarray}}\mathfrak{m}_{\lambda_i-\lambda_j}$ and $U_{mn}=\displaystyle\bigoplus\limits_{\begin{subarray}{c} \tilde{l}_{m-1}< i\leq \tilde{l}_m\\
			\tilde{l}_{n-1}< j\leq \tilde{l}_n\end{subarray}}\mathfrak{m}_{\lambda_i+\lambda_j},$
	\end{center}
	with $1\leq n<m\leq r$ if $\alpha_l\not\in\Theta$ and $1\leq n<m\leq r-1$ if $\alpha_l\in\Theta.$ For each $(m,n)$, $W_{mn}$ and $U_{mn}$ are equivalent. We denote $M_{mn}=W_{mn}\oplus U_{mn}.$\\
	
	$c)$
	\begin{center}
		$M_{rn}=\displaystyle\bigoplus\limits_{\begin{subarray}{c} \tilde{l}_{r-1}<i\leq \tilde{l}_r\\
			\tilde{l}_{n-1}< j\leq \tilde{l}_n\end{subarray}}\mathfrak{m}_{\lambda_i-\lambda_j}\oplus\mathfrak{m}_{\lambda_i+\lambda_j},$
	\end{center}
	with $1\leq n\leq r-1,$ if $\alpha_l\in\Theta.$ All these subspaces are pairwise inequivalent.\\
	
	$d)$
	\begin{center}
		$U_i=\vspan\{u_{\tilde{l}_{i-1}+s,\tilde{l}_{i-1}+s}-u_{\tilde{l}_{i-1}+s+1,\tilde{l}_{i-1}+s+1}:1\leq s\leq l_i-1\}\cup\{u_{\tilde{l}_{i-1}+s,\tilde{l}_{i-1}+t}:1\leq t<s\leq l_i\},$
	\end{center}
	for $i$ such that $l_i>1$ and $1\leq i\leq r$ if $\alpha_l\not\in\Theta,$ $1\leq i\leq r-1$ if $\alpha_l\in\Theta.$ All these subspaces are not equivalent. \hfill $\qed$
\end{pps}
For $l=4$, in addition to the subspaces described in the proposition above, we have more equivalent subspaces for some subsets $\Theta.$
\subsection{Flags of $D_l=\mathfrak{so}(l,l)$} The Lie algebra $\mathfrak{so}(l,l)$ is the set
\begin{equation*}
\left\{\left(\begin{array}{cc}
A & B \\
C & -A^T\end{array}\right)\in\mathfrak{gl}(2l,\mathbb{R})\ : \ B+B^T=C+C^T=\textbf{0}\right\}.
\end{equation*}
In this case
\begin{equation*}
\mathfrak{a}=\left\{\left(\begin{array}{cc}
\Lambda & \textbf{0} \\
\textbf{0} & -\Lambda\end{array}\right)\in\mathfrak{so}(l,l)\ : \ \Lambda=\diag(a_1,...,a_l)\right\}.
\end{equation*}
The roots are $\pm(\lambda_i-\lambda_j),\ \pm(\lambda_i+\lambda_j),\ 1\leq i<j\leq l$ where
\begin{equation*}
\lambda_i:\mathfrak{a}\longrightarrow \mathbb{R};\ \lambda_i\left(\left(\begin{array}{cc}
\Lambda & \textbf{0} \\
\textbf{0} & -\Lambda
\end{array}\right)\right)=a_i,\ i=1,...,l
\end{equation*}
and the set of simple roots is given by $\Sigma=\{\alpha_1,...,\alpha_l\}$, where $\alpha_i=\lambda_i-\lambda_{i+1},\ i=1,...,l-1$ and $\alpha_l=\lambda_{l-1}+\lambda_l$. The subalgebra $\mathfrak{k}$ is the set of skew-symmetric matrices in $\mathfrak{so}(l,l)$, i.e.,
\begin{equation*}
\mathfrak{k}=\left\{\left(\begin{array}{cc}
A & B \\
B & A\end{array}\right)\in\mathfrak{so}(l,l)\ : \ A+A^T=B+B^T=\textbf{0}\right\}.
\end{equation*}
Fix the $\Ad(K)-$invariant inner product $(\cdot,\cdot)$ in $\mathfrak{k}$ defined by 
\begin{equation}\label{ProductD}
\left(\left(\begin{array}{cc}A & B\\ B & A \end{array}\right),\left(\begin{array}{cc}C & D\\ D & C \end{array}\right)\right)=\displaystyle -\frac{\Tr(AC)+\Tr(BD)}{2}
\end{equation}
and consider the matrices
\begin{equation}\label{BasisD}
\begin{array}{l}
w_{ij}=E_{ij}-E_{ji}+E_{l+i,l+j}-E_{l+j,l+i},\\
u_{ij}=E_{l+i,j}-E_{l+j,i}+E_{i,l+j}-E_{j,l+i},\ 1\leq j<i\leq l,
\end{array}
\end{equation}
where $E_{ij}$ is the $2l\times 2l$ matrix with value equal to 1 in the $(i,j)-$entry and zero elsewhere, so that
\begin{equation*}
\mathfrak{m}_{\lambda_i-\lambda_j}=\vspan\{w_{ij}\},\ \mathfrak{m}_{\lambda_i+\lambda_j}=\vspan\{u_{ij}\},\ 1\leq i<j\leq l.
\end{equation*} 
Again, for each $\Theta\subseteq\Sigma$, there exist $l_1,...,l_r$ such that $l_1+...+l_r=l$ and 
\begin{equation*}
\Theta=\bigcup\limits_{l_i>1}\{\alpha_{\tilde{l}_{i-1}+1},...,\alpha_{\tilde{l}_i-1}\} \text{ or } \bigcup\limits_{l_i>1}\{\alpha_{\tilde{l}_{i-1}+1},...,\alpha_{\tilde{l}_i-1}\}\cup\{\alpha_l\},
\end{equation*} 
where $\tilde{l}_0=0$ and $\tilde{l}_i=\tilde{l}_{i-1}+l_i$, $i=1,...,r.$
\begin{pps}\label{IsotropyD}\label{p1.4.1}(\cite{PSM}) Let $\mathbb{F}_\Theta$ be a flag manifold of $D_l,$ $l\geq 5.$ Then the following subspaces are $K_\Theta-$invariant and irreducible:	
	
	$a)$\begin{center}
		$W_{mn}=\displaystyle\bigoplus\limits_{\begin{subarray}{c}
			\tilde{l}_{m-1}<i\leq\tilde{l}_m\\
			\tilde{l}_{n-1}<j\leq\tilde{l}_n
			\end{subarray}}\mathfrak{m}_{\lambda_i-\lambda_j},\ U_{mn}=\displaystyle\bigoplus\limits_{\begin{subarray}{c}
			\tilde{l}_{m-1}<i\leq\tilde{l}_m\\
			\tilde{l}_{n-1}<j\leq\tilde{l}_n
			\end{subarray}}\mathfrak{m}_{\lambda_i+\lambda_j},$
	\end{center}
	where $1\leq n<m\leq r$ if $\alpha_l\notin\Theta$, $1\leq n<m\leq r-1$ if $\alpha_l,\alpha_{l-1}\in\Theta$ and $1\leq n<m\leq r-2$ if $\alpha_l\in\Theta$ and $\alpha_{l-1}\notin\Theta.$ Also, $W_{mn}$ is equivalent to $U_{mn}$ for each $(m,n)$.
	
	$b)$\begin{center}
		$U_i=\displaystyle\bigoplus\limits_{\begin{subarray}{c}
			\tilde{l}_{i-1}<t<s\leq\tilde{l}_{i}\end{subarray}}\mathfrak{m}_{\lambda_s+\lambda_t},$
	\end{center}
	with $l_i>1,$ $1\leq i\leq r$ if $\alpha_l\notin\Theta,$ $1\leq i\leq r-1$ if $\alpha_l,\alpha_{l-1}\in\Theta$ and $1\leq i\leq r-2$ if $\alpha_l\in\Theta$ and $\alpha_{l-1}\notin\Theta.$ These subspaces are pairwise inequivalent.
	
	$c)$\begin{center}
		$M_{rn}=\displaystyle\bigoplus\limits_{\begin{subarray}{c} 
			\tilde{l}_{r-1}<i\leq\tilde{l}_r\\
			\tilde{l}_{n-1}<j\leq\tilde{l}_n\\
			\end{subarray}}\mathfrak{m}_{\lambda_i-\lambda_j}\oplus\mathfrak{m}_{\lambda_i+\lambda_j},$ 
	\end{center}
	where $1\leq n\leq r-1$ when $\alpha_l\in\Theta$ and $\alpha_{l-1}\in\Theta.$ These subspaces are  pairwise inequivalent.
	
	$d)$\begin{center}
		$M_n=\displaystyle\bigoplus\limits_{\begin{subarray}{c} 
			\tilde{l}_{r-2}<i\leq\tilde{l}_{r-1}\\
			\tilde{l}_{n-1}<j\leq\tilde{l}_n\\
			\end{subarray}}\mathfrak{m}_{\lambda_i-\lambda_j}\oplus\bigoplus\limits_{\begin{subarray}{c}
			\tilde{l}_{n-1}<j\leq\tilde{l}_n
			\end{subarray}}\mathfrak{m}_{\lambda_l+\lambda_j},\  N_n=\displaystyle\bigoplus\limits_{\begin{subarray}{c} 
			\tilde{l}_{r-2}<i\leq\tilde{l}_{r-1}\\
			\tilde{l}_{n-1}<j\leq\tilde{l}_n\\
			\end{subarray}}\mathfrak{m}_{\lambda_i+\lambda_j}\oplus\bigoplus\limits_{\begin{subarray}{c}
			\tilde{l}_{n-1}<j\leq\tilde{l}_n
			\end{subarray}}\mathfrak{m}_{\lambda_l-\lambda_j},$
	\end{center}
	with $1\leq n\leq r-2$ when $\alpha_l\in\Theta$ and $\alpha_{l-1}\notin\Theta.$ For each $n\in\{1,...,r-2\}$, $M_n$ is equivalent to $N_n$ and we set $S_n=M_n\oplus N_n.$
	
	$e)$\begin{center}
		$V_{r-1}=\displaystyle\bigoplus\limits_{\begin{subarray}{c} 
			\tilde{l}_{r-2}<t<s\leq\tilde{l}_{r-1}
			\end{subarray}}\mathfrak{m}_{\lambda_s+\lambda_t}\oplus\bigoplus\limits_{\begin{subarray}{c}
			\tilde{l}_{r-2}<t\leq\tilde{l}_{r-1}
			\end{subarray}}\mathfrak{m}_{\lambda_l-\lambda_t},$
	\end{center}
	when $\alpha_l\in\Theta$ and $\alpha_{l-1}\notin\Theta.$ \hfill $\qed$
\end{pps}
The case $l=4$ is different form the general case since we can find new $K_\Theta-$invariant subspaces, for example, if we consider the flag $\mathbb{F}_{\{\alpha_1,\alpha_2,\alpha_3\}}$ of $\mathfrak{so}(4,4,)$, the adjoint representation of $K_{\{\alpha_1,\alpha_2,\alpha_3\}}$ in $\mathfrak{m}_{\{\alpha_1,\alpha_2,\alpha_3\}}$ 
decomposes into two non-equivalent $K_{\{\alpha_1,\alpha_2,\alpha_3\}}-$invariant irreducible subspaces given by
\begin{center}
	$T_1=\vspan\{u_{21}+u_{43},u_{31}-u_{42},u_{41}+u_{32}\}$ and $S_1=\vspan\{u_{43}-u_{21},u_{31}+u_{42},u_{41}-u_{32}\}.$
\end{center}
\subsection{Flags with two or three isotropy summands} In this section, we classify all generalized real flag manifolds $\mathbb{F}_\Theta$ whose isotropy representation decomposes into two or three irreducible submodules.
\begin{pps} Suppose that $\mathbb{F}_\Theta$ is a real flag manifold of a classical Lie algebra whose isotropy representation decomposes into two or three irreducible submodules. Then $\mathbb{F}_\Theta$ is one of the manifolds in Table 1. 
\end{pps}
\begin{proof} Assume that $\mathbb{F}_\Theta$ is a flag of $A_l,$ $l\neq 3$ and let $l_1,...,l_r$ be positive integers such that $l_1+...+l_r=l+1$ and 
	\begin{equation*}
	\displaystyle\Theta=\bigcup\limits_{l_i>1}\left\{\alpha_{\tilde{l}_{i-1}+1},...,\alpha_{\tilde{l}_i-1}\right\},
	\end{equation*}
	where $\tilde{l}_0=0$, $\tilde{l}_i=\tilde{l}_{i-1}+l_i,$ $i=1,...,r$. By Proposition \ref{IsotropyA}, the isotropy representation of $\mathbb{F}_\Theta$ decomposes into the $r(r-1)/2$ irreducible submodules $M_{mn},$ $1\leq n\leq r$ defined in \eqref{SummandsA}. Therefore, it has two or three isotropy summands if and only if $r(r-1)/2=2$ or $r(r-1)/2=3$ respectively. Evidently, there exists no positive integer $r$ satisfying $r(r-1)/2=2$ and for $r=3$ we have that $r(r-1)/2=3(3-1)/2=3,$ in which case, $K_\Theta=S(O(l_1)\times O(l_2)\times O(l_3))$ and  \begin{equation*}\mathbb{F}_\Theta=K/K_\Theta=SO(l+1)/S(O(l_1)\times O(l_2)\times O(l_3)).
	\end{equation*}
	Now suppose that $\mathbb{F}_\Theta$ is a flag of $B_l,$ $l\geq 5.$ Take $l_1,...,l_r$ such that  
	\begin{equation*}
	\Theta=\bigcup\limits_{l_i>1}\{\alpha_{\tilde{l}_{i-1}+1},...,\alpha_{\tilde{l}_i-1}\} \text{ or } \bigcup\limits_{l_i>1}\{\alpha_{\tilde{l}_{i-1}+1},...,\alpha_{\tilde{l}_i-1}\}\cup\{\alpha_l\}.
	\end{equation*} 
	where $l=l_1+...+l_r$ and $\tilde{l}_{0}=0,\ \tilde{l}_i=\tilde{l}_{i-1}+l_i,$ $i=1,...,r.$ If $\alpha_l\notin \Theta,$ the irreducible $K_\Theta-$invariant subspaces of $\mathfrak{m}_\Theta$ are given by
	\begin{equation*}
	V_i,\ i=1,...,r,\ W_{mn},\ U_{mn},\ 1\leq n<m\leq r,\ U_{j},\ l_j>1,\ 1\leq j\leq r
	\end{equation*}
	where $V_i\ W_{mn},\ U_{mn}$ and $U_{j}$ are defined in Proposition \ref{IsotropyB}. Hence, we have $r+r(r-1)+h$ isotropy summands, where $h$ is the number of indices $j$ such that $l_j>1.$ If $r\geq 3$ then $r+r(r-1)+h\geq 3+6+h\geq 9.$ If $r=2,$ then $r+r(r-1)+h\geq 2+2+h\geq 4.$ If $r=1$, then $h=1$ and $r+r(r-1)+h=2,$ in this case $\Theta=\{\alpha_1,...,\alpha_{l-1}\}$ and $\mathbb{F}_\Theta\stackrel{\text{\scriptsize diff.}}{\approx}(SO(l)\times SO(l+1))/SO(l).$
	
	When $\alpha_l\in\Theta,$ the isotropy representation splits into the irreducible submodules
	\begin{equation*}
	(V_i)_1,\ (V_i)_2,\ i=1,...,r-1,\ W_{mn},\ U_{mn},\ 1\leq n<m\leq r-1,\ U_{j},\ l_j>1,\ 1\leq j\leq r-1,
	\end{equation*}
	so we have $2(r-1)+(r-1)(r-2)+h=(r-1)r+h$ summands. Assume that $r\geq 3.$ Then $(r-1)r+h\geq 6+h\geq6.$ When $r=1$ we obtain the degenerated case $\Theta=\Sigma$ and when $r=2$ we have $(r-1)r+h=2+h.$ If $h=0$ we have that $l_1=1$, $\Theta=\{\alpha_2,...,\alpha_l\}$ and $\mathbb{F}_\Theta\stackrel{\text{\scriptsize diff.}}{\approx}(SO(l)\times SO(l+1))/(SO(l-1)\times SO(l)).$ If $h=1$ then $l_1>1,$  $\Theta=\Sigma-\{\alpha_{l_1}\}$ and $\mathbb{F}_\Theta\stackrel{\text{\scriptsize diff.}}{\approx}(SO(l)\times SO(l+1))/(SO(l_1)\times SO(l_2)\times SO(l_2+1))$, or, equivalently, $\mathbb{F}_\Theta\stackrel{\text{\scriptsize diff.}}{\approx}(SO(l)\times SO(l+1))/(SO(d)\times SO(l-d)\times SO(l-d+1)),$ where $d=l_1.$\\
	
	\noindent We can proceed analogously to obtain $\mathbb{F}_\Theta\stackrel{\text{\scriptsize diff.}}{\approx}U(l)/O(l),\ U(l)/(O(1)\times U(l-1))$ or $U(l)/(O(d)\times U(l-d))$ for a flag of $C_l,\ l\geq 5$ and $\mathbb{F}_\Theta\stackrel{\text{\scriptsize diff.}}{\approx}(SO(l)\times SO(l))/S(O(l-1)\times O(1))$ for a flag of $D_l,\ l\geq 5$. The cases $A_3,\ B_2,\ B_3,\ B_4,\ C_4,$ and $D_4$ are treated case by case depending on the new invariant subspaces they have and the equivalences between them.
\end{proof}

\section{Invariant Einstein metrics}
This section is dedicated to the study of invariant Einstein metrics on the manifolds of Table 1.
\subsection{Flags with two isotropy summands} As we have seen, real flag manifolds of classical Lie groups with two isotropy summands are $SO(4)/S(O(2)\times O(2)),\ (SO(l)\times SO(l+1))/(SO(l-1)\times SO(l)),\ l\geq 3,\ (SO(l)\times SO(l+1))/SO(l),\ l\geq 3,\ l\neq 4,\ U(l)/O(l),\ l\geq3,\ U(l)/(O(1)\times U(l-1)),\ l\geq 3.$ We shall analyze each case separately. The complex case was studied by Arvanitoyeorgos and Chrysikos in \cite{AC1}.
\subsubsection{$SO(4)/S(O(2)\times O(2))$}  Let $\mathfrak{g}=\mathfrak{sl}(4,\mathbb{R})$ and $\Theta=\{\alpha_1,\alpha_3\}$,  then, the associated flag manifold $\mathbb{F}_{\{\alpha_1,\alpha_3\}}$ is diffeomorphic to $SO(4)/S(O(2)\times O(2))$. As before, we fix the invariant inner product $(\cdot,\cdot)$ given by the negative of the Killing form of $\mathfrak{so}(4)$ and the $(\cdot,\cdot)-$orthogonal basis $\{w_{ij}:1\leq j<i\leq 4\}.$ The Lie algebra of $K_{\{\alpha_1,\alpha_3\}}=S(O(2)\times O(1)\times O(1))$ is $\mathfrak{k}_{\{\alpha_1,\alpha_3\}}=\vspan\{w_{21},w_{43}\}$ and its $(\cdot,\cdot)-$orthogonal complement $\mathfrak{m}_{\{\alpha_1,\alpha_3\}}=\vspan\{w_{31},w_{32},w_{41},w_{42}\}$ decomposes into the two $K_{\{\alpha_1,\alpha_3\}}-$invariant, irreducible and inequivalent subspaces
\begin{equation*}
	M_1=\vspan\{w_{31}-w_{42},w_{41}+w_{32}\}\ \text{and}\ M_2=\vspan\{w_{31}+w_{42},w_{41}-w_{32}\}.
\end{equation*}
Therefore, every invariant metric on $A$ can be written in the basis $\mathcal{B}_{\{\alpha_1,\alpha_3\}}=\{w_{31}-w_{42},w_{41}+w_{32},w_{31}+w_{42},w_{41}-w_{32}\}$ as
\begin{equation}\label{metric1}
[A]_{\mathcal{B}_{\{\alpha_1,\alpha_3\}}}=\left(\begin{array}{cccc}
\mu_1 & 0 & 0 & 0 \\
0 & \mu_1 & 0 & 0 \\
0 & 0 & \mu_2 & 0 \\
0 & 0 & 0 & \mu_2 \\
\end{array}\right),\ \ \ \mu_1,\mu_2>0.
\end{equation}
\begin{pps} Let $A$ be an invariant metric on $SO(4)/S(O(2)\times O(2))$ written as above. Then $A$ is an Einstein metric if and only if $A$ is normal, i.e., $\mu_1=\mu_2.$
\end{pps}
\begin{proof} The vectors
	\begin{equation}\label{4.1.2}
	X_1=\displaystyle\frac{w_{31}-w_{42}}{\sqrt{8\mu_1}},\ X_2=\frac{w_{41}+w_{32}}{\sqrt{8\mu_1}},\ Y_1=\frac{w_{31}+w_{42}}{\sqrt{8\mu_2}},\ Y_2=\frac{w_{41}-w_{32}}{\sqrt{8\mu_2}}
	\end{equation}
	form an $A-$orthonormal basis of $\mathfrak{m}_\Theta$. It is easy to verify that $[X,Y]_{\mathfrak{m}_{\{\alpha_1,\alpha_3\}}}=0$ for $X$ and $Y$ in the basis \eqref{4.1.2}, in particular, $Z=0$ and, therefore, by \eqref{ric} we have that 
	\begin{equation*}
	\begin{array}{rlrllll}
	r_1 & := & \Ric(X_1,X_1) & = & \Ric(X_2,X_2) & = & \displaystyle \frac{1}{2\mu_1}\\
	\\
	r_2 & := & \Ric(Y_1,Y_1) & = & \Ric(Y_2,Y_2) & = &  \displaystyle\frac{1}{2\mu_2}
	\end{array}.
	\end{equation*}
	So $A$ is Einstein if and only if $r_1=r_2$, i.e., $\mu_1=\mu_2.$
\end{proof}
\subsubsection{$(SO(l)\times SO(l+1))/(SO(l-1)\times SO(l)),\ l\geq 3$}
Let $\mathfrak{g}=\mathfrak{so}(l+1,l),\ l\geq 3$ and $\Theta=\{\alpha_2,...,\alpha_l\}$ so that $\mathbb{F}_{\{\alpha_2,...,\alpha_l\}}\stackrel{\text{\scriptsize diff.}}{\approx}(SO(l)\times SO(l+1))/(SO(l-1)\times SO(l)).$ Fix the $\Ad(K)-$invariant inner product $(\cdot,\cdot)$ in \eqref{ProductB} and the $(\cdot,\cdot)-$orthonormal basis \eqref{BasisB}. By Proposition \ref{IsotropyB}  we have that $\mathfrak{m}_\Theta=(V_1)_1\oplus(V_1)_2$, where $(V_1)_1=\vspan\{w_{s1}-u_{s1}:2\leq s\leq l\}$ and $(V_1)_2=\vspan\{v_1,w_{s1}+u_{s1}:2\leq s\leq l\}.$ This two subspaces are $K_\Theta-$invariant, irreducible and inequivalent, thus every invariant metric $A$ (with respect to the inner product $(\cdot,\cdot)$) is determined by two positive numbers $\rho$, $\mu$ such that
\begin{equation}\label{metric2}
A\left|_{(V_1)_1}=\rho I_{(V_1)_1}\right.\ \text{and}\  A\left|_{(V_1)_2}=\mu I_{(V_1)_2}\right..
\end{equation}
\begin{pps} The invariant metric $A$ above is an Einstein metric if and only if $\mu=\left(\frac{l-1}{l-2}\right)\rho.$
\end{pps}
\begin{proof}
	Consider the $A-$orthonomal basis 
	\begin{equation*}
	F_s=\displaystyle\frac{w_{s1}-u_{s1}}{\sqrt{2\rho}},\ G_s=\frac{w_{s1}+u_{s1}}{\sqrt{2\mu}},\ Y=\frac{v_1}{\sqrt{\mu}},\ s=2,...,l.
	\end{equation*}
	Then 
	\begin{center}
		$\displaystyle[Y,F_s]_{\mathfrak{m}_\Theta}=[Y,G_s]_{\mathfrak{m}_\Theta}=[F_s,G_t]_{\mathfrak{m}_\Theta}=[F_s,F_t]_{\mathfrak{m}_\Theta}=[G_s,G_t]_{\mathfrak{m}_\Theta}=0,\ s,t=2,...,l.$
	\end{center}
	By \eqref{ric} we have that
	\begin{equation*}
	\begin{array}{rlrllll}
	r_1 & := & \Ric(F_s,F_s) & = & \displaystyle \frac{2(l-2)}{\rho}, & & s=2,...,l\\
	\\
	r_2 & := & \Ric(Y,Y) & = & \Ric(G_s,G_s) & = &  \displaystyle\frac{2(l-1)}{\mu},\ s=2,...,l,
	\end{array}
	\end{equation*}
	so $A$ is an Einstein metric if and only if $\frac{2(l-2)}{\rho}=\frac{2(l-1)}{\mu}$ or, equivalently, $\mu=\left(\frac{l-1}{l-2}\right)\rho.$
\end{proof}
\subsubsection{$(SO(l)\times SO(l+1))/SO(l),\ l\geq 3,\ l\neq 4$} The flag $\mathbb{F}_\Theta=(SO(l)\times SO(l+1))/SO(l)$ is obtained when we consider $\mathfrak{g}=\mathfrak{so}(l+1,l)$ and  $\Theta=\{\alpha_1,...,\alpha_{l-1}\}$. For $l\geq 3$ and $l\neq 4,$ the isotropy representation of $\mathbb{F}_\Theta$ splits into two inequivalent, irreducible submodules given by $V_1=\vspan\{v_1,...,v_l\}$ and $U_1=\vspan\{u_{st}:1\leq t<s\leq l\},$ where $v_j,u_{st}$ are defined in \eqref{BasisB}. Every invariant metric operator $A$ with respect to the inner product \ref{ProductB} has the form
\begin{equation}\label{metric3}
A\left|_{V_1}=\mu I_{V_1}\right.,\ A\left|_{U_1}=\gamma I_{U_1}\right., \mu,\gamma>0.
\end{equation}
\begin{pps}\label{p4.3}
	Let $A$ be the invariant metric on $(SO(l)\times SO(l+1))/SO(l)$ given by the parameters $\mu,\gamma>0.$ Then $A$ is an Einstein metric if and only $\mu=\frac{\gamma}{2}$ or $\mu=\left(\frac{l}{2l-4}\right)\gamma.$
\end{pps}
\begin{proof}
	An $A-$orthonormal basis for $\mathfrak{m}_\Theta$ is given by
	\begin{equation*}
	\displaystyle Y_{st}=\frac{u_{st}}{\sqrt{\gamma}},\ Z_j=\frac{v_j}{\sqrt{\mu}},\ 1\leq t<s\leq,\ j=1,...,l,
	\end{equation*}
	which satisfies following relations
	\begin{align*}
	&[Z_s,Z_t]_{\mathfrak{m}_\Theta}=\displaystyle-\frac{\sqrt{\gamma}}{\mu}Y_{st},\ [Z_s,Y_{st}]_{\mathfrak{m}_\Theta}=\displaystyle\frac{Z_t}{\sqrt{\gamma}},\ [Z_t,Y_{st}]_{\mathfrak{m}_\Theta}=\displaystyle-\frac{Z_s}{\sqrt{\gamma}},\ 1\leq t<s\leq l,\\
	\\
	&[Y_{st},Y_{ij}]_{\mathfrak{m}_\Theta}=0.
	\end{align*}
	Using formula \eqref{ric} we obtain
	\begin{center}
		$\begin{array}{ccccl}
		r_1 & := & \Ric(Z_j,Z_j) & = & \displaystyle\frac{2(l-1)}{\mu}-\frac{(l-1)\gamma}{2\mu^2},\ j=1,...,l \\
		\\
		r_2 & := & \Ric(Y_{st},Y_{st}) & = & \displaystyle\frac{2(l-2)}{\gamma}+\frac{\gamma}{2\mu^2},\ 1\leq t<s\leq l.
		\end{array}$
	\end{center}
	Therefore, $A$ is an Einstein metric if and only if 
	\begin{align*}
	\displaystyle\frac{2(l-1)}{\mu}-\frac{(l-1)\gamma}{2\mu^2}=\frac{2(l-2)}{\gamma}+\frac{\gamma}{2\mu^2}&\Longleftrightarrow \displaystyle-\frac{2(l-1)}{\mu}+\frac{2(l-2)}{\gamma}+\frac{l\gamma}{2\mu^2}=0\\
	\\
	&\Longleftrightarrow-4(l-1)\gamma\mu+4(l-2)\mu^2+l\gamma^2=0\\
	\\
	&\Longleftrightarrow l(2\mu-\gamma)^2-4\mu(2\mu-\gamma)=0\\
	\\
	&\Longleftrightarrow (2\mu-\gamma)(l(2\mu-\gamma)-4\mu)=0\\
	\\
	&\Longleftrightarrow\mu=\frac{\gamma}{2}\ \text{or}\ \mu=\left(\frac{l}{2l-4}\right)\gamma.
	\end{align*}
\end{proof}
\subsubsection{$U(l)/O(l),\ l\geq 3$}
Let $\mathfrak{g}=\mathfrak{sp}(l,\mathbb{R}),\ l\geq 3$ and $\Theta=\{\alpha_1,...,\alpha_{l-1}\}$, then $\mathbb{F}_\Theta\stackrel{\text{\scriptsize diff.}}{\approx}U(l)/O(l).$ We consider the product $(\cdot,\cdot)$ defined in \eqref{ProductC} and the $(\cdot,\cdot)-$orthogonal basis \eqref{BasisC}. Proposition \ref{IsotropyC} implies that the isotropy representation of $U(l)/O(l)$ decomposes into the two inequivalent, irreducible submodules $V_1=\vspan\{u_{11}+...+u_{ll}\}$ and $U_1=\vspan\{u_{11}-u_{22},...,u_{l-1,l-1}-u_{ll}\}\cup\{u_{st}:1\leq t<s\leq l\}$. Hence, very invariant metric $A$ has the form
\begin{equation}\label{metric4}
A|_{V_1}=\mu^{(0)} I_{V_1},\ A|_{U_1}=\mu^{(1)} I_{U_1}, \mu^{(0)},\mu^{(1)}>0.
\end{equation}
\begin{pps} There is no $U(l)-$invariant Einstein metric on $U(l)/O(l).$
\end{pps}
\begin{proof} Let $A$ be an invariant metric as in \eqref{metric4} and consider any $A-$orthonormal basis $\{X_i\}$ of $\mathfrak{m}_\Theta.$ Each matrix $X_i$ has the form
	\begin{center}
		$X_i=\left(\begin{array}{cc}
		\textbf{0} & -B_i\\
		B_i & \textbf{0}
		\end{array}\right)$, $B_i-B_i^T=0,$
	\end{center}
	therefore 
	\begin{center}
		$[X_i,X_j]_{\mathfrak{m}_\Theta}=\left[\left(\begin{array}{cc}
		\textbf{0} & -B_i\\
		B_i & \textbf{0}
		\end{array}\right),\left(\begin{array}{cc}
		\textbf{0} & -B_j\\
		B_j & \textbf{0}
		\end{array}\right)\right]_{\mathfrak{m}_\Theta}=\left(\begin{array}{cc}
		B_jB_i-B_iB_j & \textbf{0} \\
		\textbf{0} & B_jB_i-B_iB_j
		\end{array}\right)_{\mathfrak{m}_\Theta}=\textbf{0}$
	\end{center}
	for every $i,j$. Take $Z_1=u_{11}+...+u_{ll}\in V_1$ and $Y=u_{21}\in U_1$, by formula \eqref{ric} we have
	\begin{center}
		$\begin{array}{lcccl}
		\Ric(Z_1,Z_1) & := & \displaystyle-\frac{1}{2}\langle Z_1,Z_1\rangle & = & 0\\
		\\
		\Ric(Y,Y) & := & \displaystyle-\frac{1}{2}\langle Y,Y\rangle & = & 2l.
		\end{array}$
	\end{center}
	Since $2l\neq0$, then $A$ cannot be an Einstein metric. 
\end{proof}
\begin{obs} In \cite{Kerr}, Böhm and Kerr proved that every compact simply connected homogeneous space up to dimension 11 admits at least one invariant Einstein metric. For $l=3$ and $l=4,$ the manifold $U(l)/O(l)$ has dimension 6 and 10 respectively, this is not a contradiction since these manifolds are not simply connected (see \cite{W}).	
\end{obs}
\subsubsection{$U(l)/(O(1)\times U(l-1)),\ l\geq 3$}
Let $\mathfrak{g}=\mathfrak{sp}(l,\mathbb{R}),\ l\geq 3$ and $\Theta=\{\alpha_2,...,\alpha_l\}$ so that $\mathbb{F}_\Theta\stackrel{\text{\scriptsize diff.}}{\approx}U(l)/(O(1)\times U(l-1))$. In this case, we consider $(\cdot,\cdot)$ as in \eqref{ProductC} and the basis \eqref{BasisC} of $\mathfrak{k}\cong \mathfrak{u}(l)$. By Proposition \ref{IsotropyC}, we have that $\mathfrak{m}_\Theta=V_1\oplus M_{21}$, where $V_1=\vspan\{u_{11}\}$ and $M_{21}=\vspan\{w_{s1},u_{s1}:s=2,...,l\}$ are not equivalent. Every invariant metric $A$ has the form
\begin{equation}\label{metric5}
A|_{V_1}=\mu^{(0)} I_{V_1},\ A|_{M_{21}}=\mu^{(21)} I_{M_{21}}, \mu^{(0)},\mu^{(21)}>0.
\end{equation}
\begin{pps} The metric \eqref{metric5} is an Einstein metric if and only if $\mu^{(0)}=2\mu^{(21)}.$
\end{pps}
\begin{proof} An $A-$orthonormal basis of $\mathfrak{m}_\Theta$ is given by 
	\begin{equation*}\label{4.1.11}
	X_{s1}=\frac{w_{s1}}{\sqrt{\mu^{(21)}}},\ Y_{s1}=\frac{u_{s1}}{\sqrt{\mu^{(21)}}},\ Z_1=\sqrt{\frac{2}{\mu^{(0)}}}u_{11},\ s=2,...,l  
	\end{equation*}
	and satisfy
	\begin{align*}
	&[Z_1,X_{s1}]_{\mathfrak{m}_\Theta}=-\sqrt{\frac{2}{\mu^{(0)}}}Y_{s1},\ [Z_1,Y_{s1}]_{\mathfrak{m}_\Theta}=\sqrt{\frac{2}{\mu^{(0)}}}X_{s1},\ [X_{s1},Y_{s1}]_{\mathfrak{m}_\Theta}=-\frac{\sqrt{2\mu^{(0)}}}{\mu^{(21)}}Z_1,\ s=2,...,l,\\
	&\left[X_{s1},X_{t1}\right]_{\mathfrak{m}_\Theta}=[X_{s1},Y_{t1}]_{\mathfrak{m}_\Theta}=[Y_{s1},Y_{t1}]_{\mathfrak{m}_\Theta}=0,\ s\neq t.
	\end{align*}
	As before, we can apply \eqref{ric} to obtain
	\begin{center}
		$\begin{array}{cllll}
		\Ric(Z_1,Z_1) & = & \displaystyle \frac{(l-1)\mu^{(0)}}{(\mu^{(21)})^2}, & &\\
		\\
		\Ric(X_{s1},X_{s1}) & = & \Ric(Y_{s1},Y_{s1}) & = &  \displaystyle\frac{2l}{\mu^{(21)}}-\frac{\mu^{(0)}}{(\mu^{(21)})^2},\ s=2,...,l,
		\end{array}$
	\end{center}
	so $A$ is an Einstein metric if and only if 
	\begin{align*}
	\displaystyle\frac{(l-1)\mu^{(0)}}{(\mu^{(21)})^2}=\frac{2l}{\mu^{(21)}}-\frac{\mu^{(0)}}{(\mu^{(21)})^2}&\Longleftrightarrow \displaystyle(l-1)\mu^{(0)}=2l\mu^{(21)}-\mu^{(0)}\\
	\\
	&\Longleftrightarrow l\mu^{(0)}=2l\mu^{(21)}\\
	\\
	&\Longleftrightarrow \mu^{(0)}=2\mu^{(21)}.
	\end{align*}
\end{proof}
\subsubsection{$(SO(4)\times SO(4))/SO(4)$}
The manifold $\mathbb{F}_\Theta=(SO(4)\times SO(4))/SO(4)$ is a flag of $\mathfrak{so}(l,l)$ obtained when $\Theta=\{\alpha_1,\alpha_2,\alpha_3\}$ or $\{\alpha_1,\alpha_2,\alpha_4\}$. Let us assume that $\Theta=\{\alpha_1,\alpha_2,\alpha_3\},$ then the isotropy representation decomposes into two inequivalent $SO(4)-$invariant irreducible subspaces given by
\begin{equation*}
	T_1=\vspan\{u_{21}+u_{43},u_{31}-u_{42},u_{41}+u_{32}\}\ \text{and}\ S_1=\vspan\{u_{43}-u_{21},u_{31}+u_{42},u_{41}-u_{32}\},
\end{equation*}
where the matrices $u_{ij}$ are defined in \eqref{BasisD}. Since $T_1$ and $S_1$ are not equivalent, we have that every metric operator $A$ with respect to the inner product \eqref{ProductD} is determined by positive real numbers $\mu_1,\mu_2$ such that 
\begin{equation}\label{metric6}
A\left|_{T_1}=\mu_1 I_{T_1}\right.,\ A\left|_{S_1}=\mu_2 I_{S_1}\right..
\end{equation}
\begin{pps}\label{p4.1.6}
	Let $A$ be an invariant metric on $(SO(4)\times SO(4))/SO(4)$. Then $A$ is an Einstein metric if and only if $A$ is normal. 
\end{pps}
\begin{proof}
	In this case, an $A-$orthonormal basis of $\mathfrak{m}_{\{\alpha_1,\alpha_2,\alpha_3\}}$ is given by
	\begin{equation*}
	X_1=\displaystyle\frac{u_{21}+u_{43}}{\sqrt{2\mu_1}},\ X_2=\displaystyle\frac{u_{31}-u_{42}}{\sqrt{2\mu_1}},\ X_3=\displaystyle\frac{u_{41}+u_{32}}{\sqrt{2\mu_1}},
	\end{equation*}
	\begin{equation*}
	Y_1=\displaystyle\frac{u_{43}-u_{21}}{\sqrt{2\mu_1}},\ Y_2=\displaystyle\frac{u_{31}+u_{42}}{\sqrt{2\mu_1}},\ Y_3=\displaystyle\frac{u_{41}-u_{32}}{\sqrt{2\mu_1}}
	\end{equation*}
	and we have that
	\begin{center}
		$[X_i,X_j]_{\mathfrak{m}_{\{\alpha_1,\alpha_2,\alpha_3\}}}=[Y_i,X_j]_{\mathfrak{m}_{\{\alpha_1,\alpha_2,\alpha_3\}}}=[Y_i,Y_j]_{\mathfrak{m}_{\{\alpha_1,\alpha_2,\alpha_3\}}}=0$, $i,j=1,2,3.$
	\end{center}
	Therefore
	\begin{center}
		$\left\{\begin{array}{ccccl}
		\Ric(X_i,X_i) & = &\displaystyle-\frac{1}{2}\langle X_i,X_i\rangle & = & \displaystyle\frac{2}{\mu_1}\\
		\\
		\Ric(Y_i,Y_i) & = &\displaystyle-\frac{1}{2}\langle Y_i,Y_i\rangle & = & \displaystyle\frac{2}{\mu_2}\\
		\end{array},\ i=1,2,3.\right.$
	\end{center}
	So $A$ is an Einstein metric if and only if $\frac{2}{\mu_1}=\frac{2}{\mu_2}$, i.e., $\mu_1=\mu_2$ ($A$ is normal).
\end{proof}

\noindent For $\Theta=\{\alpha_1,\alpha_2,\alpha_4\}$ we consider the automorphism $\eta$ of $\mathfrak{so}(4)\oplus\mathfrak{so}(4)$ given by
\begin{center}
	$\eta(w_{ij})=u_{ij},\ \eta(u_{ij})=w_{ij},\ 1\leq j<i\leq 3$\\
	\ \\
	$\eta(w_{4j})=u_{4j},\ \eta(u_{4j})=w_{4j},\ j=1,2,3.$
\end{center}
Observe that $\eta(\mathfrak{m}_{\{\alpha_1,\alpha_2,\alpha_4\}})=\mathfrak{m}_{\{\alpha_1,\alpha_2,\alpha_3\}}$ and $\eta$ maps isotropy summands into isotropy summands, therefore, every invariant metric on $\mathfrak{m}_{\{\alpha_1,\alpha_2,\alpha_4\}}$ has the form $\left(\eta|_{\mathfrak{m}_{\{\alpha_1,\alpha_2,\alpha_4\}}}\right)^*A,$ where $A$ is an invariant metric on $\mathfrak{m}_{\{\alpha_1,\alpha_2,\alpha_3\}}.$ Using the formula \eqref{ric}, it is easy to see that $\eta$ preserves the Ricci components, thus, an invariant metric on $\mathbb{F}_{\{\alpha_1,\alpha_2,\alpha_4\}}$ is Einstein if and only if it is normal.
\subsection{Flags with three isotropy summands} Homogeneous Einstein metrics on complex generalized flag manifolds with three isotropy summands were completely classified in \cite{A} and \cite{K}. Here we study the Einstein equation for real flag manifolds of classical type whose isotropy representation decomposes into three isotropy summands.
\subsubsection{$SO(4)/S(O(2)\times O(1)\times O(1))$ }\label{ss4.2.1}
The flag $SO(4)/S(O(2)\times O(1)\times O(1))$ is obtained when $\mathfrak{g}=\mathfrak{sl}(4,\mathbb{R})$ and $\Theta=\{\alpha_1\},\ \{\alpha_2\}$ or $\{\alpha_3\}.$ Let us normalize the $\Ad(K)-$invariant inner product $(\cdot,\cdot)$ by setting
\begin{center}
	$g_0=\frac{1}{4}(\cdot,\cdot)=-\frac{1}{4}\langle\cdot,\cdot\rangle.$
\end{center}
Suppose that $\Theta=\{\alpha_1\}$, then the Lie algebra of $K_{\{\alpha_1\}}$ is the subalgebra $\mathfrak{k}_{\{\alpha_1\}}=\vspan\{w_{21}\}$ and, by Proposition \ref{IsotropyA}, the adjoint representation of $K_{\{\alpha_1\}}$ on $\mathfrak{m}_{\{\alpha_1\}}$ decomposes into the irreducible subrepresentations  $M_{32}=\vspan\{w_{43}\},$ $M_{21}=\vspan\{w_{31},w_{32}\}$ and $M_{31}=\vspan\{w_{42},w_{41}\}$, where $M_{21}$ and $M_{31}$ are equivalent. This results in the existence of non-diagonal invariant metrics, more precisely, we have the following proposition:
\begin{pps}\label{non-diagonalA}(\cite{GGN})
Every invariant metric $A$ on $SO(4)/S(O(2)\times O(1)\times O(1))$ is written in the ordered basis $\mathcal{B}=\{w_{43},w_{31},w_{32},w_{42},w_{41}\}$ in the following form:
\begin{equation}
\label{metric7}
[A]_{\mathcal{B}}=\left(\begin{array}{ccccc}
\mu_0 & 0 & 0 & 0 & 0 \\
0 & \mu_1 & 0 & b & 0 \\
0 & 0 & \mu_1 & 0 & -b \\
0 & b & 0 & \mu_2 & 0 \\
0 & 0 & -b & 0 & \mu_2 \\
\end{array}\right),\ \mu_0,\mu_1,\mu_2>0.
\end{equation}
\end{pps} 
\begin{proof}
Since $M_{21}$ is equivalent to $M_{31}$, we have that any metric operator $A$ is written in $\mathcal{B}$ in the form
	
	\begin{equation*}[A]_{\mathcal{B}}=\left(\begin{array}{ccccc}
		\mu_0 & 0 & 0 & 0 & 0 \\
		0 & \mu_1 & 0 & b & d \\
		0 & 0 & \mu_1 & c & e \\
		0 & b & c & \mu_2 & 0 \\
		0 & d & e & 0 & \mu_2 \\
		\end{array}\right),\ \mu_0,\mu_1,\mu_2>0,
	\end{equation*}
where the matrix 
\begin{equation*}
\left(\begin{array}{cc}
b & c\\
d & e
\end{array}\right)
\end{equation*}
defines an intertwining map between $M_{21}$ and $M_{31}.$ Given $k\in K_{\{\alpha_1\}}=S(O(2)\times O(1)\times O(1)),$ we have that
	\begin{equation*}
		k=\left(\begin{array}{cccc}
		r & s & 0 & 0 \\
		t & u & 0 & 0 \\
		0 & 0 & v & 0 \\
		0 & 0 & 0 & z \\
		\end{array}\right),\ v,z\in\{\pm1\},\ \left(\begin{array}{cc}
		r & s\\
		t & u
		\end{array}\right)\in O(2)\ \text{and}\ \det(k)=1
	\end{equation*}
so that the map $\Ad(k)|_{\mathfrak{m}_{\{\alpha_1\}}}$ is written in the basis $\mathcal{B}$ in the form
	\begin{center}$\left[\Ad(k)|_{\mathfrak{m}_{\{\alpha_1\}}}\right]_{\mathcal{B}}=\left(\begin{array}{ccccc}
		vz & 0 & 0 & 0 & 0 \\
		0 & vr & vs & 0 & 0 \\
		0 & vt & vu & 0 & 0 \\
		0 & 0 & 0 & zu & zt \\
		0 & 0 & 0 & zs & zr \\
		\end{array}\right).$
	\end{center}
Using the fact that $A$ commutes with $\Ad(k)|_{\mathfrak{m}_{\{\alpha_1\}}}$ and taking suitable matrices $k$ we can conclude that $b=-e$ and $c=d=0,$ so we have the result.
\end{proof}
\begin{pps}\label{p4.7}
	Let $A$ be an invariant metric on the flag $\mathbb{F}_{\{\alpha_1\}}$ of $\mathfrak{sl}(4,\mathbb{R})$ written as in \eqref{metric7}. Then $A$ is an Einstein metric if and only if its entries satisfy one of the following conditions:
	\begin{itemize}
		\item[(E1)] $b=0,$ $\mu_1=\mu_2=\frac{3}{4}\mu_0$
		\item[(E2)] $b>0,$ $b=\mu_1=\frac{\mu_2}{3}=\frac{\mu_0}{2}$
		\item[(E3)] $b>0,$ $b=\mu_2=\frac{\mu_1}{3}=\frac{\mu_0}{2}$
		\item[(E4)] $b<0,$ $-b=\mu_1=\frac{\mu_2}{3}=\frac{\mu_0}{2}$
		\item[(E5)] $b<0,$ $-b=\mu_2=\frac{\mu_1}{3}=\frac{\mu_0}{2}$
	\end{itemize}
\end{pps}
\begin{proof}We separate two cases:
	
$\bullet$ $A$ diagonal ($b=0$):
	
The vectors
	\begin{equation*}
	\displaystyle X_0=\frac{w_{43}}{\sqrt{\mu_0}},\ X_1=\frac{w_{31}}{\sqrt{\mu_1}},\ X_2=\frac{w_{32}}{\sqrt{\mu_1}},\ X_3=\frac{w_{42}}{\sqrt{\mu_2}},\ X_4=\frac{w_{41}}{\sqrt{\mu_2}}
	\end{equation*}
	form an $A-$orthonormal basis of $\mathfrak{m}_{\{\alpha_1\}}.$ The non-zero bracket relations between these vectors are given by
	\begin{equation*}
	\begin{array}{lclclcl}
	[X_0,X_1]_{\mathfrak{m}_{\{\alpha_1\}}} & = & \sqrt{\frac{\mu_2}{\mu_0\mu_1}}X_4,& &[X_0,X_2]_{\mathfrak{m}_{\{\alpha_1\}}} & = & \sqrt{\frac{\mu_2}{\mu_0\mu_1}}X_3,\\
	\\
	
	[X_0,X_3]_{\mathfrak{m}_{\{\alpha_1\}}} & = & -\sqrt{\frac{\mu_1}{\mu_0\mu_2}}X_2,& &[X_0,X_4]_{\mathfrak{m}_{\{\alpha_1\}}} & = & -\sqrt{\frac{\mu_1}{\mu_0\mu_2}}X_1,\\
	\\
	
	[X_1,X_4]_{\mathfrak{m}_{\{\alpha_1\}}} & = & \sqrt{\frac{\mu_0}{\mu_1\mu_2}}X_0,& &[X_2,X_3]_{\mathfrak{m}_{\{\alpha_1\}}} & = & \sqrt{\frac{\mu_0}{\mu_1\mu_2}}X_0.\\
	\end{array}
	\end{equation*}
	Since $[X_i,X_j]$ is $A-$orthogonal to $X_i$ and $X_j$ for all $i,j\in\{0,1,2,3,4\}$, then, equation \eqref{U} implies 
	\begin{center}
		$Z=\displaystyle\sum\limits_iU(X_i,X_i)=0.$
	\end{center}
	A direct application of the formula \eqref{ric} gives us the components of the Ricci tensor:
	\begin{equation*}
	\begin{array}{rlrllll}
	& & r_0 & := & \Ric(X_0,X_0) & = & \displaystyle \frac{2}{\mu_0}+\frac{\mu_0}{\mu_1\mu_2}-\frac{\mu_2}{\mu_0\mu_1}-\frac{\mu_1}{\mu_0\mu_2}\\
	\\
	r_1 & := & \Ric(X_1,X_1) & = & \Ric(X_2,X_2) & = & \displaystyle \frac{2}{\mu_1}+\frac{\mu_1}{2\mu_0\mu_2}-\frac{\mu_2}{2\mu_0\mu_1}-\frac{\mu_0}{2\mu_1\mu_2}\\
	\\
	r_2 & := & \Ric(X_3,X_3) & = & \Ric(X_4,X_4) & = &  \displaystyle\frac{2}{\mu_2}+\frac{\mu_2}{2\mu_0\mu_1}-\frac{\mu_1}{2\mu_0\mu_2}-\frac{\mu_0}{2\mu_1\mu_2}.
	\end{array}
	\end{equation*}
	Thus, Einstein condition for the diagonal metric $A$ reduces to $r_0=r_1=r_2,$ which gives us
	\begin{equation*}
	\mu_1=\mu_2=\frac{3}{4}\mu_0.
	\end{equation*} 
	$\bullet$ $A$ non-diagonal  ($b\neq0$):
	
The eigenvalues of $A$ are
	\begin{center}
		$\xi_0=\mu_0,\ \ \xi_1=\displaystyle\frac{\mu_1+\mu_2-\sqrt{4b^2+\left(\mu_1-\mu_2\right)^2}}{2},\ \ \xi_2=\displaystyle\frac{\mu_1+\mu_2+\sqrt{4b^2+\left(\mu_1-\mu_2\right)^2}}{2}$
	\end{center}
	and satisfy the relations
	\begin{equation*}
	\begin{array}{lcl}
	&&\xi_1+\xi_2=\mu_1+\mu_2, \ \ \ \xi_1\xi_2=\mu_1\mu_2-b^2,\\
	\\
	b^2 & = & (\xi_2-\mu_2)(\xi_2-\mu_1)=(\mu_1-\xi_1)(\xi_2-\mu_1)\\
	\\
	& = & (\mu_2-\xi_1)(\mu_1-\xi_1)=(\mu_2-\xi_1)(\xi_2-\mu_2).
	\end{array}
	\end{equation*}
	Since $A$ is positive definite then $\xi_1,\xi_2>0$. Let us set 
	\begin{center}$\begin{array}{l}
		c_1=\xi_1((\xi_1-\mu_2)^2+b^2)=\xi_1(\xi_1-\mu_2)(\xi_1-\xi_2) \text{ and}\\
		\\
		c_2=\xi_2((\xi_2-\mu_2)^2+b^2)=\xi_2(\xi_2-\mu_2)(\xi_2-\xi_1).\end{array}$
	\end{center}
	Then the vectors
	\begin{equation*}
	\begin{array}{cl}
	\displaystyle X_0=\frac{w_{43}}{\sqrt{\xi_0}}, & \displaystyle X_1=\frac{(\xi_1-\mu_2)w_{31}+bw_{42}}{\sqrt{c_1}},\ X_2=\frac{(\xi_2-\mu_1)w_{32}+bw_{41}}{\sqrt{c_1}},\\ 
	\\
	& \displaystyle X_3=\frac{(\xi_2-\mu_2)w_{31}+bw_{42}}{\sqrt{c_2}},\ X_4=\frac{(\xi_1-\mu_1)w_{32}+bw_{41}}{\sqrt{c_2}}
	\end{array}
	\end{equation*}
	form an $A-$orthonormal basis for $\mathfrak{m}_{\{\alpha_1\}}.$ The non-zero bracket relations are given by
	\begin{align*}
	&[X_0,X_1]_{\mathfrak{m}_{\{\alpha_1\}}}= \frac{b}{\xi_1-\xi_2}\left(\frac{2}{\sqrt{\xi_0}}X_2+\sqrt{\frac{\xi_2}{\xi_0\xi_1}}\left(\frac{\mu_2-\mu_1}{|b|}\right)X_4\right),\\
	\\
	&[X_0,X_2]_{\mathfrak{m}_{\{\alpha_1\}}}= \frac{b}{\xi_2-\xi_1}\left(\frac{2}{\sqrt{\xi_0}}X_1+\sqrt{\frac{\xi_2}{\xi_0\xi_1}}\left(\frac{\mu_2-\mu_1}{|b|}\right)X_3\right),\\
	\\
	&[X_0,X_3]_{\mathfrak{m}_{\{\alpha_1\}}}= \frac{b}{\xi_2-\xi_1}\left(\sqrt{\frac{\xi_1}{\xi_0\xi_2}}\left(\frac{\mu_1-\mu_2}{|b|}\right)X_2+\frac{2}{\sqrt{\xi_0}}X_4\right),\\
	\\
	&[X_0,X_4]_{\mathfrak{m}_{\{\alpha_1\}}}= \frac{b}{\xi_1-\xi_2}\left(\sqrt{\frac{\xi_1}{\xi_0\xi_2}}\left(\frac{\mu_1-\mu_2}{|b|}\right)X_1+\frac{2}{\sqrt{\xi_0}}X_3\right),\\
	\\
	&[X_1,X_2]_{\mathfrak{m}_{\{\alpha_1\}}}= \frac{2b\sqrt{\xi_0}}{\xi_1(\xi_1-\xi_2)}X_0,\\
	\\
	&[X_1,X_4]_{\mathfrak{m}_{\{\alpha_1\}}}= \frac{b(\mu_2-\mu_1)}{|b|(\xi_1-\xi_2)}\sqrt{\frac{\xi_0}{\xi_1\xi_2}}X_0,\\
	\\
	&[X_2,X_3]_{\mathfrak{m}_{\{\alpha_1\}}}= \frac{b(\mu_2-\mu_1)}{|b|(\xi_2-\xi_1)}\sqrt{\frac{\xi_0}{\xi_1\xi_2}}X_0,\\
	\\
	&[X_3,X_4]_{\mathfrak{m}_{\{\alpha_1\}}}= \frac{2b\sqrt{\xi_0}}{\xi_2(\xi_2-\xi_1)}X_0.\\
	\end{align*}
	Observe that $[X_i,X_j]$ is $A-$orthogonal to $X_i$ and $X_j$ for all $i,j\in\{0,1,2,3,4\}$, thus, $Z=0.$ If $A$ is an Einstein metric, then $\Ric(X_1,X_3)=0$ (because $\Ric=cg$). By \eqref{ric} we have
	\begin{align*}
	\Ric(X_1,X_3)=& -\frac{1}{2}g\left([X_1,X_2]_{\mathfrak{m}_{\{\alpha_1\}}},[X_3,X_2]_{\mathfrak{m}_{\{\alpha_1\}}}\right)-\frac{1}{2}g\left([X_1,X_4]_{\mathfrak{m}_{\{\alpha_1\}}},[X_3,X_4]_{\mathfrak{m}_{\{\alpha_1\}}}\right)\\
	\\
	&-\frac{1}{2}g\left([X_1,X_0]_{\mathfrak{m}_{\{\alpha_1\}}},[X_3,X_0]_{\mathfrak{m}_{\{\alpha_1\}}}\right)-\frac{1}{2}\langle X_1,X_3\rangle\\
	\\
	&+\frac{1}{2}g\left([X_0,X_2]_{\mathfrak{m}_{\{\alpha_1\}}},X_1\right)g\left([X_0,X_2]_{\mathfrak{m}_{\{\alpha_1\}}},X_3\right)\\
	\\
	& +\frac{1}{2}g\left([X_0,X_4]_{\mathfrak{m}_{\{\alpha_1\}}},X_1\right)g\left([X_0,X_4]_{\mathfrak{m}_{\{\alpha_1\}}},X_3\right)\\
	\\	
	=& -\frac{1}{2}\left(\frac{2b\sqrt{\xi_0}}{\xi_1(\xi_1-\xi_2)}\right)\left(\frac{b(\mu_1-\mu_2)}{|b|(\xi_2-\xi_1)}\sqrt{\frac{\xi_0}{\xi_1\xi_2}}\right)-\frac{1}{2}\left(\frac{2b\sqrt{\xi_0}}{\xi_2(\xi_2-\xi_1)}\right)\left(\frac{b(\mu_2-\mu_1)}{|b|(\xi_1-\xi_2)}\sqrt{\frac{\xi_0}{\xi_1\xi_2}}\right)\\
	\\
	& -\frac{1}{2}\left(\frac{2b}{\sqrt{\xi_0}(\xi_2-\xi_1)}\right)\left(\frac{b(\mu_2-\mu_1)}{|b|(\xi_2-\xi_1)}\sqrt{\frac{\xi_1}{\xi_0\xi_2}}\right)-\frac{1}{2}\left(\frac{2b}{\sqrt{\xi_0}(\xi_2-\xi_1)}\right)\left(\frac{b(\mu_2-\mu_1)}{|b|(\xi_1-\xi_2)}\sqrt{\frac{\xi_2}{\xi_0\xi_1}}\right)\\
	\\
	&+\frac{1}{2}\left(\frac{2b}{\sqrt{\xi_0}(\xi_2-\xi_1)}\right)\left(\frac{b(\mu_2-\mu_1)}{|b|(\xi_2-\xi_1)}\sqrt{\frac{\xi_2}{\xi_0\xi_1}}\right)+\frac{1}{2}\left(\frac{2b}{\sqrt{\xi_0}(\xi_1-\xi_2)}\right)\left(\frac{b(\mu_1-\mu_2)}{|b|(\xi_1-\xi_2)}\sqrt{\frac{\xi_1}{\xi_0\xi_2}}\right)\\
	\\
	=&\ \frac{b^2(\mu_2-\mu_1)}{|b|(\xi_2-\xi_1)^2}\left(-\frac{\xi_0}{\xi_1\sqrt{\xi_1\xi_2}}+\frac{\xi_0}{\xi_2\sqrt{\xi_1\xi_2}}-\frac{2\sqrt{\xi_1}}{\xi_0\sqrt{\xi_2}}+\frac{2\sqrt{\xi_2}}{\xi_0\sqrt{\xi_1}}\right)\\
	\\
	=&\ \frac{b^2(\mu_2-\mu_1)}{|b|(\xi_2-\xi_1)^2}\left(\frac{-\xi_0^2(\xi_2-\xi_1)+2\xi_1\xi_2(\xi_2-\xi_1)}{\xi_0\xi_1\xi_2\sqrt{\xi_1\xi_2}}\right)\\
	\\
	=&\ \frac{|b|(\mu_2-\mu_1)(2\xi_1\xi_2-\xi_0^2)}{\xi_0(\xi_1\xi_2)^{\frac{3}{2}}(\xi_2-\xi_1)},
	\end{align*}
	so, a necessary condition for $A$ to be an Einstein metric is that $\mu_1=\mu_2$ or $2\xi_1\xi_2-\xi_0^2=0.$
	
\textit{Case 1.} $\mu:=\mu_1=\mu_2.$
	
We can use formula \eqref{ric} to obtain
	\begin{align*}
		&r_0 = \Ric(X_0,X_0) = \displaystyle\frac{\xi_0(\xi_1\xi_2+2b^2)}{(\xi_1\xi_2)^2},\\
		\\
		&r_1 = \Ric(X_1,X_1) = \Ric(X_2,X_2) =  \displaystyle\frac{4\xi_1-\xi_0}{2\xi_1^2},\\
		\\
		&r_2 = \Ric(X_3,X_3) = \Ric(X_4,X_4) =  \displaystyle\frac{4\xi_2-\xi_0}{2\xi_2^2}.\\
	\end{align*}
	If $r_0=r_1=r_2,$ then $\xi_1=\xi_2$, which is not possible since $b\neq 0$. Therefore, we have no solutions in this case.
	
\textit{Case 2.} $2\xi_1\xi_2-\xi_0^2=0.$
	
The components of the Ricci tensor are $r_0=\Ric(X_0,X_0),$ $r_1=\Ric(X_1,X_1)=\Ric(X_2,X_2)$ and $r_2=\Ric(X_3,X_3)=\Ric(X_4,X_4).$ By \eqref{ric} we have
	\begin{align*}
	r_1-r_2=&\  \frac{\xi_1\xi_2(\mu_2-\mu_1)^2}{\xi_0(\xi_2-\xi_1)^2}\left(\frac{1}{\xi_2^2}-\frac{1}{\xi_1^2}\right)+\frac{2b^2\xi_0}{(\xi_2-\xi_1)^2}\left(\frac{1}{\xi_2^2}-\frac{1}{\xi_1^2}\right)-2\left(\frac{1}{\xi_2}-\frac{1}{\xi_1}\right)\\
	\\
	=&\  \left(\frac{1}{\xi_2^2}-\frac{1}{\xi_1^2}\right)\left(\frac{\xi_1\xi_2(\mu_2-\mu_1)^2(\xi_1+\xi_2)+2b^2\xi_0^2(\xi_1+\xi_2)-2\xi_1\xi_2\xi_0(\xi_2-\xi_1)^2}{\xi_0(\xi_2-\xi_1)^2(\xi_1+\xi_2)}\right)\\
	\\
	=&\  \left(\frac{1}{\xi_2^2}-\frac{1}{\xi_1^2}\right)\left(\frac{\xi_1\xi_2(\xi_1+\xi_2)((\mu_2-\mu_1)^2+4b^2)-2\xi_1\xi_2\xi_0(\xi_2-\xi_1)^2}{\xi_0(\xi_2-\xi_1)^2(\xi_1+\xi_2)}\right)\\
	\\
	=&\  \left(\frac{1}{\xi_2^2}-\frac{1}{\xi_1^2}\right)\left(\frac{\xi_1\xi_2(\xi_1+\xi_2)(\xi_2-\xi_1)^2-2\xi_1\xi_2\xi_0(\xi_2-\xi_1)^2}{\xi_0(\xi_2-\xi_1)^2(\xi_1+\xi_2)}\right)\\
	\\
	=&\  \left(\frac{1}{\xi_2^2}-\frac{1}{\xi_1^2}\right)\left(\frac{\xi_1\xi_2(\xi_1+\xi_2-2\xi_0)}{\xi_0(\xi_1+\xi_2)}\right)\\
	\\
	=&\  \frac{(\xi_1-\xi_2)(\xi_1+\xi_2-2\sqrt{2\xi_1\xi_2})}{\sqrt{2}(\xi_1\xi_2)^{\frac{3}{2}}}\\
	\end{align*}
	and \begin{equation*}
		\begin{array}{lcl}
		\displaystyle r_0-\frac{r_1+r_2}{2} & = &
		\displaystyle\frac{b^2}{(\xi_2-\xi_1)^2}\left(\frac{3\xi_0}{\xi_1^2}+\frac{3\xi_0}{\xi_2^2}-\frac{8}{\xi_0}\right)+\frac{(\mu_2-\mu_1)^2}{2(\xi_2-\xi_1)^2}\left(\frac{3\xi_0}{\xi_1\xi_2}-\frac{2\xi_2}{\xi_0\xi_1}-\frac{2\xi_1}{\xi_0\xi_2}\right)\\
		\\
		& & \displaystyle-\left(\frac{2}{\xi_0}-\frac{1}{\xi_1}-\frac{1}{\xi_2}\right)\\
		\\
		&= & \displaystyle \frac{10\sqrt{2}b^2+3\sqrt{2}\xi_1\xi_2-\sqrt{2}(\xi_1-\xi_2)^2-2(\xi_1+\xi_2)\sqrt{\xi_1\xi_2}}{(\xi_1\xi_2)^{\frac{3}{2}}}.
		\end{array}
	\end{equation*}
	Since the systems of equations
	\begin{center}
		$\left\{\begin{array}{l}
		r_0-r_1=0\\
		\\
		r_1-r_2=0
		\end{array}\right.$ \hspace{3cm}	$\left\{\begin{array}{l}
		r_1-r_2=0\\
		\\
		r_0-\frac{r_1+r_2}{2}=0
		\end{array}\right.$ 
	\end{center}
	are equivalent, we have that $A$ is an Einstein metric if and only if $\xi_1$, $\xi_2$ satisfy the equations
	\begin{center}
		$\left\{\begin{array}{l}
		10\sqrt{2}b^2+3\sqrt{2}\xi_1\xi_2-\sqrt{2}(\xi_1-\xi_2)^2-2(\xi_1+\xi_2)\sqrt{\xi_1\xi_2}=0\\
		\\
		\xi_1+\xi_2-2\sqrt{2\xi_1\xi_2}=0.
		\end{array}\right.$
	\end{center}
	By solving this system for $\xi_1<\xi_2$, we obtain the solutions:
	\begin{center}
		$\left\{\begin{array}{l}
		\xi_1=(-2+\sqrt{2})b\\
		\xi_2=(-2-\sqrt{2})b
		\end{array}\right.,$ $b<0,$ \hspace{1cm}	$\left\{\begin{array}{l}
		\xi_1=(2-\sqrt{2})b\\
		\xi_2=(2+\sqrt{2})b
		\end{array}\right.,$ $b>0,$
	\end{center}
	or, equivalently, 
	\begin{center}
		$\begin{array}{ll}
		\left\{\begin{array}{l}
		\mu_1=-3b\\
		\mu_2=-b
		\end{array}\right.,\ b<0, \hspace{2cm}& \left\{\begin{array}{l}
		\mu_1=-b\\
		\mu_2=-3b
		\end{array}\right.,\ b<0,\\
		\\
		\left\{\begin{array}{l}
		\mu_1=b\\
		\mu_2=3b
		\end{array}\right.,\ b>0, \hspace{2cm}&	\left\{\begin{array}{l}
		\mu_1=3b\\
		\mu_2=b
		\end{array}\right.,\ b>0.
		\end{array}$
	\end{center}
	In all these cases, $\mu_0=\xi_0=\sqrt{2\xi_1\xi_2}=\sqrt{4b^2}=2|b|.$ 
\end{proof}
When $\Theta=\{\alpha_2\}$ or $\{\alpha_3\}$, consider the maps 
\begin{center}
	$\varphi_i:\mathbb{F}_{\{\lambda_i-\lambda_{i+1}\}}\longrightarrow\mathbb{F}_{\{\lambda_1-\lambda_2\}};\ \ \varphi_i\left(kK_{\{\lambda_i-\lambda_{i+1}\}}\right)=e_ike_i^TK_{\{\lambda_1-\lambda_2\}},\ i=1,2,$
\end{center}
where
\begin{center} $e_1=\left(\begin{array}{cccc}
	0 & 1 & 0 & 0\\
	0 & 0 & 1 & 0\\
	1 & 0 & 0 & 0\\
	0 & 0 & 0 & 1\\
	\end{array}\right)$ and $e_2=\left(\begin{array}{cccc}
	0 & 0 & 1 & 0\\
	0 & 0 & 0 & 1\\
	1 & 0 & 0 & 0\\
	0 & 1 & 0 & 0\\
	\end{array}\right).$
\end{center}
These maps are diffeomorphisms. It is easy to show that every invariant metric on $\mathbb{F}_{\{\alpha_i\}}$ has the form $\varphi_i^*g,$ where $g$ is an invariant metric on $\mathbb{F}_{\{\alpha_1\}}.$ By \cite[Lemma 7.2]{Lee} we have that the Ricci tensor associated to $\varphi_i^*g$ is equal to the pull-back by $\varphi_i$ of the Ricci tensor associated to $g$, thus, $\varphi_i^*g$ is an Einstein metric if and only if is so $g$. Therefore, Einstein metrics on $\mathbb{F}_{\{\alpha_i\}}$ are obtained by taking the pullback by $\varphi_i$ of the Einstein metrics on $\mathbb{F}_{\{\alpha_1\}}.$
\subsubsection{$SO(l+1)/S(O(l_1)\times O(l_2)\times O(l_3)),\ l\geq 2,\ l\neq3,\ l_1+l_2+l_3=l+1$} We shall present a more general result about invariant Einstein metrics on homogeneous spaces with three isotropy summands; it will be useful to study Einstein metrics on our particular case. Let $G$ be a compact connected Lie group, $H$ a closed subgroup of $G$ and let $\mathfrak{g},$ $\mathfrak{h}$ be the Lie algebras of $G$, $H$ respectively. Assume that the isotropy representation of $G/H$ is decomposed into non-equivalent three irreducible components and consider an $(\cdot,\cdot)-$orthogonal reductive decomposition $\mathfrak{g}=\mathfrak{h}\oplus\mathfrak{m},$ where $(\cdot,\cdot)$ is an $\Ad(G)-$invariant inner product on $\mathfrak{g}.$ Let $\mathfrak{m}=\mathfrak{m}_1\oplus\mathfrak{m}_2\oplus\mathfrak{m}_3$ be the irreducible decomposition of $\mathfrak{m}.$ Each invariant metric $A$ with respect to $(\cdot,\cdot)$ on $G/H$ can be represented by positive numbers $x_1,$ $x_2,$ $x_3$ such that $A|_{\mathfrak{m}_i}=x_i\text{I}_{\mathfrak{m}_i},$ $i=1,2,3.$ Let $\mathcal{M}$ be the set of all $G-$invariant metrics on $G/H$ with volume 1, i.e.,
\begin{equation*}
\mathcal{M}=\{A=(x_1,x_2,x_3)\in(\mathbb{R}^+)^3:x_1^{dim\ \mathfrak{m}_1}x_2^{dim\ \mathfrak{m}_2}x_3^{dim\ \mathfrak{m}_3}=\frac{1}{v_0^2}\}
\end{equation*} 
where $v_0=Vol(G/H,(\cdot,\cdot)|_{\mathfrak{m}\times\mathfrak{m}}).$ Denote by $S(A)$ the scalar curvature of $(G/H,A).$ By \cite[Corollary 7.39]{Bes} we have the following formula for the scalar curvature:
\begin{equation}\label{Scalar} 
S(A)=\displaystyle-\frac{1}{4}\sum\limits_{i,j}|[X_i,X_j]_{\mathfrak{m}}|^2-\frac{1}{2}\sum\limits{i}\langle X_i,X_i\rangle-|Z|^2
\end{equation}
where $|\cdot|$ is the norm with respect to $A$, $\{X_i\}$ is an $A-$orthonormal basis of $\mathfrak{m}$ and $Z=\sum\limits_{i}U(X_i,X_i).$ The following result gives us a tool to study existence of invariant Einstein metrics.

\begin{pps}\label{Propositionscalar}(\cite{WZ}) $A\in\mathcal{M}$ is an Einstein metric if and only if 
	\begin{equation*}
	\displaystyle\frac{\partial S}{\partial u}(A)=\frac{\partial S}{\partial v}(A)=0,
	\end{equation*}
	where $u=\displaystyle\frac{x_2}{x_1}$ and $v=\displaystyle\frac{x_3}{x_1}.$ \hfill $\qed$
\end{pps}
For the flag $\mathbb{F}_{\Sigma-\{\alpha_{d_1},\alpha_{d_2}\}},$ $1\leq d_1<d_2\leq l$ of $\mathfrak{sl}(l+1,\mathbb{R}),$ $l\neq 3$ the isotropy representation decomposes into the non-equivalent irreducible submodules $M_{21}=\vspan\{w_{st}:s=d_1+1,...,d_2,\ t=1,...,d_1\},$ $M_{31}=\vspan\{w_{st}:s=d_2+1,...,l+1,\ t=1,...,d_1\}$ and $M_{32}=\vspan\{w_{st}:s=d_2+1,...,l+1,\ t=d_1+1,...,d_2\}$. Let $l_1:=d_1,$ $l_2:=d_2-d_1$ and $l_3:=l+1-d_2$, then
\begin{equation*}
\displaystyle\Theta=\bigcup\limits_{l_i>1}\left\{\alpha_{\tilde{l}_{i-1}+1},...,\alpha_{\tilde{l}_i-1}\right\}\ \text{and}\ \mathbb{F}_\Theta=SO(l+1)/S(O(l_1)\times O(l_2)\times O(l_3)).
\end{equation*}Fix the $\Ad(K)-$invariant inner product $g_0$ on $\mathfrak{k}=\mathfrak{so}(l+1)$ given by 
\begin{equation*}
g_0=\displaystyle\frac{1}{2(l-1)}(\cdot,\cdot)=\displaystyle-\frac{1}{2(l-1)}\langle\cdot,\cdot\rangle.
\end{equation*}Any invariant metric $A$ (with respect to $g_0$) is defined by positive numbers $\mu_{21},\mu_{31},\mu_{32}$ such that
\begin{equation}\label{metric8}
A|_{M_{mn}}=\mu_{mn}I_{M_{mn}},\ 1\leq n<m\leq 3.
\end{equation}
\begin{pps}\label{p4.2.3} The flag manifold $\mathbb{F}_\Theta=SO(l+1)/S(O(l_1)\times O(l_2)\times O(l_3))$ has at most four $SO(l+1)-$invariant Einstein metrics up to homotheties. Moreover, when $l_2=l_3=:m\geq 3$ we have
	
	$a)$ If $1\leq l_1< 2\sqrt{m-1}$ or $l_1>\frac{(m-2)^2+m\sqrt{m^2-4m+8}}{2}$ then $\mathbb{F}_\Theta$ has exactly two invariant Einstein metrics up to homotheties.
	
	$b)$ If $l_1=2\sqrt{m-1}$ or $l_1=\frac{(m-2)^2+m\sqrt{m^2-4m+8}}{2}$ then $\mathbb{F}_\Theta$ has exactly three invariant Einstein metrics up to homotheties.
	
	$c)$ If $2\sqrt{m-1}<l_1<\frac{(m-2)^2+m\sqrt{m^2-4m+8}}{2}$ then $\mathbb{F}_\Theta$ has exactly four invariant Einstein metrics up to homotheties.
\end{pps}
\begin{proof} Let $A$ be an invariant metric on $SO(l+1)/S(O(l_1)\times O(l_2)\times O(l_3))$ as in \eqref{metric8}. We consider the $A-$orthonormal basis 
	\begin{align}\label{4.2.4}
	\begin{split}
	&X_{st}=\displaystyle \frac{w_{st}}{\sqrt{\mu_{21}}},\ s=l_1+1,...,l_1+l_2,\ t=1,...,l_1\\
	\\
	&X_{st}=\displaystyle\frac{w_{st}}{\sqrt{\mu_{31}}},\ s=l_1+l_2+1,...,l+1,\ t=1,...,l_1\\
	\\
	&X_{st}=\displaystyle\frac{w_{st}}{\sqrt{\mu_{32}}},\ s=l_1+l_2+1,...,l+1,\ t=l_1+1,...l_1+l_2.
	\end{split}
	\end{align}
Then, we have two cases:
	
\textit{Case 1.}  $l_2\neq l_3.$
	
Applying formula \eqref{Scalar} to basis \eqref{4.2.4} we have that
	\begin{equation*}
	S(A)=\displaystyle-\frac{l_1l_2l_3}{2}\left(\frac{\mu_{21}}{\mu_{31}\mu_{32}}+\frac{\mu_{31}}{\mu_{21}\mu_{32}}+\frac{\mu_{32}}{\mu_{21}\mu_{31}}\right)+(l-1)\left(\frac{l_1l_2}{\mu_{21}}+\frac{l_1l_3}{\mu_{31}}+\frac{l_2l_3}{\mu_{32}}\right).
	\end{equation*}
	Now, assume that $A$ has volume 1 and let $u:=\displaystyle\frac{\mu_{31}}{\mu_{21}}$, $v:=\displaystyle\frac{\mu_{32}}{\mu_{21}}$, then 
	\begin{align*}
	\frac{S(A)}{v_0^{\frac{2}{N}}}=\frac{S(u,v)}{v_0^{\frac{2}{N}}}=&\displaystyle-\frac{l_1l_2l_3}{2}\left(u^{\frac{l_1l_3}{N}-1}v^{\frac{l_2l_3}{N}-1}+u^{\frac{l_1l_3}{N}+1}v^{\frac{l_2l_3}{N}-1}+u^{\frac{l_1l_3}{N}-1}v^{\frac{l_2l_3}{N}+1}\right)\\
	\\
	&+(l-1)\displaystyle\left(l_1l_2u^{\frac{l_1l_3}{N}}v^{\frac{l_2l_3}{N}}+l_1l_3u^{\frac{l_1l_3}{N}-1}v^{\frac{l_2l_3}{N}}+l_2l_3u^{\frac{l_1l_3}{N}}v^{\frac{l_2l_3}{N}-1}\right),
	\end{align*}
	where $N:=l_1l_2+l_1l_3+l_2l_3$ and $v_0=Vol(\mathbb{F}_\Theta,g_0).$ Computing the partial deritives of $S$ we obtain
	\begin{align*}
	\displaystyle v_0^{-\frac{2}{N}}\frac{\partial S}{\partial u}=&\displaystyle-\frac{l_1l_2l_3}{2}\left(\left(\frac{l_1l_3}{N}-1\right)u^{\frac{l_1l_3}{N}-2}v^{\frac{l_2l_3}{N}-1}+\left(\frac{l_1l_3}{N}+1\right)u^{\frac{l_1l_3}{N}}v^{\frac{l_2l_3}{N}-1}+\left(\frac{l_1l_3}{N}-1\right)u^{\frac{l_1l_3}{N}-2}v^{\frac{l_2l_3}{N}+1}\right)\\
	\\
	&+(l-1)\displaystyle\left(\frac{l_1^2l_2l_3}{N}u^{\frac{l_1l_3}{N}-1}v^{\frac{l_2l_3}{N}}+l_1l_3\left(\frac{l_1l_3}{N}-1\right)u^{\frac{l_1l_3}{N}-2}v^{\frac{l_2l_3}{N}}+\frac{l_1l_2l_3^2}{N}u^{\frac{l_1l_3}{N}-1}v^{\frac{l_2l_3}{N}-1}\right),\\
	\\
	\displaystyle v_0^{-\frac{2}{N}}\frac{\partial S}{\partial v}=&\displaystyle-\frac{l_1l_2l_3}{2}\left(\left(\frac{l_2l_3}{N}-1\right)u^{\frac{l_1l_3}{N}-1}v^{\frac{l_2l_3}{N}-2}+\left(\frac{l_2l_3}{N}-1\right)u^{\frac{l_1l_3}{N}+1}v^{\frac{l_2l_3}{N}-2}+\left(\frac{l_2l_3}{N}+1\right)u^{\frac{l_1l_3}{N}-1}v^{\frac{l_2l_3}{N}}\right)\\
	\\
	&+(l-1)\displaystyle\left(\frac{l_1l_2^2l_3}{N}u^{\frac{l_1l_3}{N}}v^{\frac{l_2l_3}{N}-1}+\frac{l_1l_2l_3^2}{N}u^{\frac{l_1l_3}{N}-1}v^{\frac{l_2l_3}{N}-1}+l_2l_3\left(\frac{l_2l_3}{N}-1\right)u^{\frac{l_1l_3}{N}}v^{\frac{l_2l_3}{N}-2}\right),
	\end{align*}
	thus
	\begin{align*}
	\displaystyle d_1:=v_0^{-\frac{2}{N}}\frac{N}{l_1l_3}u^{-\frac{l_1l_3}{N}+2}v^{-\frac{l_2l_3}{N}+1}\frac{\partial S}{\partial u}=&\displaystyle-\frac{l_1l_2l_3}{2}\left(-\left(\frac{N}{l_1l_3}-1\right)+\left(\frac{N}{l_1l_3}+1\right)u^2-\left(\frac{N}{l_1l_3}-1\right)v^2\right)\\
	\\
	&+(l-1)\left(l_1l_2uv-l_1l_3\left(\frac{N}{l_1l_3}-1\right)v+l_2l_3u\right),\\
	\\
	d_2:=\displaystyle v_0^{-\frac{2}{N}}\frac{N}{l_2l_3}u^{-\frac{l_1l_3}{N}+1}v^{-\frac{l_2l_3}{N}+2}\frac{\partial S}{\partial v}=&\displaystyle-\frac{l_1l_2l_3}{2}\left(-\left(\frac{N}{l_2l_3}-1\right)-\left(\frac{N}{l_2l_3}-1\right)u^2+\left(\frac{N}{l_2l_3}+1\right)v^2\right)\\
	\\
	&+(l-1)\left(l_1l_2uv+l_1l_3v-l_2l_3\left(\frac{N}{l_2l_3}-1\right)u\right).
	\end{align*}
	If $A$ is an Einstein metric then, by Proposition \ref{Propositionscalar}, $\displaystyle\frac{\partial S}{\partial u}=\frac{\partial S}{\partial v}=0$, so $d_1=d_2=0$ and
	\begin{equation}\label{4.2.5}
	\displaystyle\left(\frac{N}{l_2l_3}+1\right)d_1+\left(\frac{N}{l_1l_3}-1\right)d_2=0\Longleftrightarrow C_1u^2+C_2uv+C_3v+C_4u+C_5=0,
	\end{equation}
	where
	\begin{center}
		$C_1=-N(l_1+l_2),\ C_2=\displaystyle\frac{(l-1)N(l_1+l_2)}{l_3},\ C_3=-\displaystyle\frac{(l-1)N(l_1+l_3)}{l_3},$\\
	\end{center}
	\begin{center}
		$C_4=\displaystyle\frac{(l-1)N(l_3-l_2)}{l_3},\ C_5=\displaystyle\frac{Nl_2(l_1+l_3)}{l_3}.$
	\end{center}
	Now, we will show that $u\neq-\frac{C_3}{C_2}.$ In fact if $u=-\frac{C_3}{C_2}$ then
	\begin{align*}
	C_1u^2+C_2uv+C_3v+C_4u+C_5=0&\Longrightarrow\displaystyle \frac{C_1C_3^2}{C_2^2}-\frac{C_4C_3}{C_2}+C_5=0\\
	\\
	&\Longrightarrow \frac{C_1C_3^2-C_4C_3C_2+C_5C_2^2}{C_2^2}=0\\
	\\
	&\Longrightarrow\displaystyle\frac{2N^3(l-1)^2(l_1+l_3)(l_1+l_2)(l_2-l_3)}{l_3^3C_2^2}=0\\
	\\
	&\Longrightarrow l_2=l_3=0,
	\end{align*}
	which is absurd. This fact allows us to isolate $v$ in \eqref{4.2.5}, obtaining
	\begin{equation}\label{4.2.6}
	v=\displaystyle\frac{-C_1u^2-C_4u-C_5}{C_2u+C_3},
	\end{equation}
	thus
	\begin{align*}
	0=(C_2u+C_3)^2d_1=&\displaystyle-\frac{l_2^2(l_1+l_3)}{2}(C_2u+C_3)^2-\frac{l_2(N+l_1l_3)}{2}u^2(C_2u+C_3)^2\\
	\\
	&\displaystyle+\frac{l_2^2(l_1+l_3)}{2}(C_1u^2+C_4u+C_5)^2-(l-1)l_1l_2u(C_2u+C_3)(C_1u^2+C_4u+C_5)\\
	\\
	&\displaystyle+(l-1)l_2(l_1+l_3)(C_2u+C_3)(C_1u^2+C_4u+C_5)+(l-1)l_2l_3u(C_2u+C_3)^2\\
	\\
	=:&f(u),
	\end{align*}
	where $f$ is a fourth-degree polynomial with real coefficients. Since $f$ has at most four positive roots and $v$ is determined by $u$ (because of \eqref{4.2.6}), then we have at most four possibilities for $(u,v)$. i.e., we have at most four invariant Einstein metrics on $SO(l+1)/S(O(l_1)\times O(l_2)\times O(l_3))$ up to homotheties.
	
\textit{Case 2.} $l_2=l_3=:m$.
	
In this case, we use the basis \eqref{4.2.4} and formula \eqref{ric} to obtain the Ricci components
	\begin{align*}
	&r_1=\displaystyle \frac{m}{2}\left(\frac{\mu_{21}}{\mu_{31}\mu_{32}}-\frac{\mu_{31}}{\mu_{21}\mu_{32}}-\frac{\mu_{32}}{\mu_{21}\mu_{31}}\right)+\frac{l-1}{\mu_{21}},\\
	\\
	&r_2=\displaystyle \frac{m}{2}\left(\frac{\mu_{31}}{\mu_{21}\mu_{32}}-\frac{\mu_{32}}{\mu_{21}\mu_{31}}-\frac{\mu_{21}}{\mu_{31}\mu_{32}}\right)+\frac{l-1}{\mu_{31}},\\
	\\
	&r_3=\displaystyle \frac{l_1}{2}\left(\frac{\mu_{32}}{\mu_{21}\mu_{31}}-\frac{\mu_{21}}{\mu_{31}\mu_{32}}-\frac{\mu_{31}}{\mu_{21}\mu_{32}}\right)+\frac{l-1}{\mu_{32}}.\\
	\end{align*}
	Since every invariant metric which is homothetic to an Einstein metric is also an Einstein metric, we can assume that $\mu_{32}=1$. The solutions in $\mathbb{C}\times \mathbb{C}$ of the system
	\begin{center}
		$\left\{\begin{array}{l}
		r_1-r_2=0\\
		\\
		r_2-r_3=0
		\end{array}\right.$
	\end{center}
	are given by
	\begin{align*}
	&\left\{\begin{array}{l}
	\mu_{21}=\displaystyle\frac{a_1+\sqrt{\Delta_1}}{4(m-1)}\\
	\\
	\mu_{31}=\displaystyle\frac{a_1+\sqrt{\Delta_1}}{4(m-1)}
	\end{array}\right.\ (1)\\
	\\
	&\left\{\begin{array}{l}
	\mu_{21}=\displaystyle\frac{a_1-\sqrt{\Delta_1}}{4(m-1)}\\
	\\
	\mu_{31}=\displaystyle\frac{a_1-\sqrt{\Delta_1}}{4(m-1)}
	\end{array}\right.\ (2)\\
	\\
	&\left\{\begin{array}{l}
	\mu_{21}=\displaystyle\frac{a_2+m\sqrt{\Delta_2}}{2m^2(m+l_1-1)}\\
	\\
	\mu_{31}=\displaystyle\frac{a_2-m\sqrt{\Delta_2}}{2m^2(m+l_1-1)}
	\end{array}\right.\ (3)\\
	\\
	&\left\{\begin{array}{l}
	\mu_{21}=\displaystyle\frac{a_2-m\sqrt{\Delta_2}}{2m^2(m+l_1-1)}\\
	\\
	\mu_{31}=\displaystyle\frac{a_2+m\sqrt{\Delta_2}}{2m^2(m+l_1-1)}
	\end{array}\right.\ (4),
	\end{align*}
	\\
	where 
	\begin{center}
		$a_1=2m+l_1-2,\ 
		a_2=m(m+l_1-1)(2m+l_1-2),\ 
		\Delta_1=l_1^2-4(m-1),$
	\end{center}
	\begin{center}
		$\Delta_2=(m+l_1-1)(-l_1^2+l_1(m-2)^2+m^3-4m^2+8m-4).$
	\end{center}
	Suppose that $m\geq 3.$ Since $l_1$ is a positive integer, we have that $\Delta_1\geq 0$ if and only if $l_1\geq 2\sqrt{m-1}$  and $\Delta_2\geq 0$ if and only if
	\begin{align*}
	-l_1^2+l_1(m-2)^2+m^3-4m^2+8m-4\geq 0 \Longleftrightarrow 1\leq l_1\leq \displaystyle\frac{(m-2)^2+m\sqrt{m^2-4m+8}}{2}.
	\end{align*}
	Also, we note that  $a_1>\sqrt{\Delta_1}$ and $a_2>m\sqrt{\Delta_2}$ when $\Delta_1,\Delta_2\geq 0$, in fact 
	\begin{align*}
	a_1>\sqrt{\Delta_1}&\Longleftrightarrow 2m+l_1-2>\sqrt{l_1^2+4(1-m)}\\
	\\
	&\Longleftrightarrow (2m+l_1-2)^2>l_1^2+4(1-m)\\
	\\
	&\Longleftrightarrow l_1^2+4(m-1)l_1+4(m-1)^2>l_1^2+4(1-m)\\
	\\
	&\Longleftrightarrow 4(m-1)(l_1+1)+4(m-1)^2>0
	\end{align*}
	which holds for $m\geq 3$, and
	\begin{align*}
	a_2>m\sqrt{\Delta_2}&\Longleftrightarrow \sqrt{m+l_1-1}(2m+l_1-2)>\sqrt{-l_1^2+l_1(m-2)^2+m^3-4m^2+8m-4}\\
	\\
	&\Longleftrightarrow (m+l_1-1)(2m+l_1-2)^2>-l_1^2+l_1(m-2)^2+m^3-4m^2+8m-4\\
	\\
	&\Longleftrightarrow l_1^3+(5m-4)l_1^2+(7m^2-12m+4)l_1+(3m^3-8m^2+4m)>0
	\end{align*}
	which is true since $5m-4,\ 7m^2-12m+4,\ 3m^3-8m^2+4m>0$ when $m\geq 3.$ Observe that $2\sqrt{m-1}<\displaystyle\frac{(m-2)^2+m\sqrt{m^2-4m+8}}{2},$ so we have:
	
$\bullet$ If $1\leq l_1< 2\sqrt{m-1}$, then solutions $(1),$ $(2)$ are complex and solutions $(3),$ $(4)$ are positive, thus, only $(3)$, $(4)$ are invariant Einstein metrics.
	
$\bullet$ If $l_1=2\sqrt{m-1}$, then $(1),$ $(2),$ $(3),$ $(4)$ are all positive solutions and $(1)=(2),$ so we have three invariant Einstein metrics.
	
$\bullet$ If $2\sqrt{m-1}<l_1<\displaystyle\frac{(m-2)^2+m\sqrt{m^2-4m+8}}{2},$ then $(1),$ $(2),$ $(3),$ $(4)$ are positive distinct solutions and we have four invariant Einstein metrics.
	
$\bullet$ If $l_1=\displaystyle\frac{(m-2)^2+m\sqrt{m^2-4m+8}}{2},$ then $(1),$ $(2),$ $(3),$ $(4)$ are all positive solutions and $(3)=(4),$ so we have three invariant Einstein metrics.
	
$\bullet$ If $l_1>\displaystyle\frac{(m-2)^2+m\sqrt{m^2-4m+8}}{2},$ then $(1),$ $(2)$ are positive and $(3),$ $(4)$ are positive, therefore, only $(1),$ $(2)$ are invariant Einstein metrics. 
\end{proof}
\subsubsection{$(SO(4)\times SO(5))/SO(4)$} Let $\mathfrak{g}=\mathfrak{so}(5,4)$ and consider the flag given by $\Theta=\{\alpha_1,\alpha_2,\alpha_3\},$ then $\mathbb{F}_\Theta\stackrel{\text{\scriptsize diff.}}{\approx}(SO(4)\times SO(5))/SO(4)$. We fix the $\Ad(SO(4)\times SO(5))-$invariant inner product $(\cdot,\cdot)$ defined in \eqref{ProductB} and the matrices $w_{ij},u_{ij},v_j$ defined in \eqref{BasisB}. The Lie algebra of $K_{\{\alpha_1,\alpha_2,\alpha_3\}}=SO(4)$ is given by 
$\mathfrak{k}_{\{\alpha_1,\alpha_2,\alpha_3\}}=\vspan\{w_{st}:1\leq t<s\leq 4\}$ and  $\mathfrak{m}_{\{\alpha_1,\alpha_2,\alpha_3\}}=V_1\oplus T_1\oplus T_2$, where $V_1=\vspan\{v_1,v_2,v_3,v_4\},$ $T_1=\vspan\{u_{21}+u_{43},u_{31}-u_{42},u_{41}+u_{32}\},$ $T_2=\vspan\{u_{21}-u_{43},u_{31}+u_{42},u_{41}-u_{32}\}$ are not equivalent. Every invariant metric $A$ with respect to $(\cdot,\cdot)$ has the form
\begin{equation}\label{metric9}
A|_{V_1}=\mu I_{V_1},\ A|_{T_1}=\gamma_1 I_{T_1},\ A|_{T_2}=\gamma_2 I_{T_2},\ \mu,\gamma_1,\gamma_2>0.
\end{equation}
\begin{pps}\label{p4.10} Let $A$ be an invariant metric on $(SO(4)\times SO(5))/SO(4)$ as in \eqref{metric9}. Then $A$ is an Einstein metric if and only if $\gamma_1=\gamma_2=:\gamma$ and $\mu=\frac{\gamma}{2}$ or $\mu=\gamma.$
\end{pps}
\begin{proof}
	We consider the $A-$orthonormal basis
	\begin{align*}
	\displaystyle &X_1=\displaystyle\frac{u_{21}+u_{43}}{\sqrt{2\gamma_1}},\ X_2=\displaystyle\frac{u_{31}-u_{42}}{\sqrt{2\gamma_1}},\ X_3=\displaystyle\frac{u_{41}+u_{32}}{\sqrt{2\gamma_1}}\\
	\\
	&Y_1=\displaystyle\frac{u_{43}-u_{21}}{\sqrt{2\gamma_2}},\ Y_2=\displaystyle\frac{u_{31}+u_{42}}{\sqrt{2\gamma_2}},\ Y_3=\displaystyle\frac{u_{41}-u_{32}}{\sqrt{2\gamma_1}},\ Z_j=\frac{v_j}{\sqrt{\mu}},\ j=1,2,3,4.
	\end{align*}
	Then
	\begin{align*}
	&[Z_1,Z_2]_{\mathfrak{m}_{\{\alpha_1,\alpha_2,\alpha_3\}}}=\displaystyle \sqrt{\frac{\gamma_1}{2\mu^2}}X_1-\sqrt{\frac{\gamma_2}{2\mu^2}}Y_1,\ [Z_1,Z_3]_{\mathfrak{m}_{\{\alpha_1,\alpha_2,\alpha_3\}}}=\displaystyle \sqrt{\frac{\gamma_1}{2\mu^2}}X_2+\sqrt{\frac{\gamma_2}{2\mu^2}}Y_2,\\
	\\
	&[Z_1,Z_4]_{\mathfrak{m}_{\{\alpha_1,\alpha_2,\alpha_3\}}}=\displaystyle \sqrt{\frac{\gamma_1}{2\mu^2}}X_3+\sqrt{\frac{\gamma_2}{2\mu^2}}Y_3,\ [Z_2,Z_3]_{\mathfrak{m}_{\{\alpha_1,\alpha_2,\alpha_3\}}}=\displaystyle \sqrt{\frac{\gamma_1}{2\mu^2}}X_3-\sqrt{\frac{\gamma_2}{2\mu^2}}Y_3,\\
	\\
	&[Z_2,Z_4]_{\mathfrak{m}_{\{\alpha_1,\alpha_2,\alpha_3\}}}=\displaystyle- \sqrt{\frac{\gamma_1}{2\mu^2}}X_2+\sqrt{\frac{\gamma_2}{2\mu^2}}Y_2,\ [Z_3,Z_4]_{\mathfrak{m}_{\{\alpha_1,\alpha_2,\alpha_3\}}}=\displaystyle \sqrt{\frac{\gamma_1}{2\mu^2}}X_1+\sqrt{\frac{\gamma_2}{2\mu^2}}Y_1,\\
	\\
	&[Z_1,X_1]_{\mathfrak{m}_{\{\alpha_1,\alpha_2,\alpha_3\}}}=\displaystyle -\frac{Z_2}{\sqrt{2\gamma_1}},\ [Z_1,X_2]_{\mathfrak{m}_{\{\alpha_1,\alpha_2,\alpha_3\}}}=\displaystyle -\frac{Z_3}{\sqrt{2\gamma_1}},\ [Z_1,X_3]_{\mathfrak{m}_{\{\alpha_1,\alpha_2,\alpha_3\}}}=\displaystyle -\frac{Z_4}{\sqrt{2\gamma_1}},\\
	\\
	&[Z_1,Y_1]_{\mathfrak{m}_{\{\alpha_1,\alpha_2,\alpha_3\}}}=\displaystyle \frac{Z_2}{\sqrt{2\gamma_2}},\ [Z_1,Y_2]_{\mathfrak{m}_{\{\alpha_1,\alpha_2,\alpha_3\}}}=\displaystyle -\frac{Z_3}{\sqrt{2\gamma_2}},\ [Z_1,Y_3]_{\mathfrak{m}_{\{\alpha_1,\alpha_2,\alpha_3\}}}=\displaystyle -\frac{Z_4}{\sqrt{2\gamma_2}},\\
	\\
	&[Z_2,X_1]_{\mathfrak{m}_{\{\alpha_1,\alpha_2,\alpha_3\}}}=\displaystyle \frac{Z_1}{\sqrt{2\gamma_1}},\ [Z_2,X_2]_{\mathfrak{m}_{\{\alpha_1,\alpha_2,\alpha_3\}}}=\displaystyle \frac{Z_4}{\sqrt{2\gamma_1}},\ [Z_2,X_3]_{\mathfrak{m}_{\{\alpha_1,\alpha_2,\alpha_3\}}}=\displaystyle -\frac{Z_3}{\sqrt{2\gamma_1}},\\
	\\
	&[Z_2,Y_1]_{\mathfrak{m}_{\{\alpha_1,\alpha_2,\alpha_3\}}}=\displaystyle -\frac{Z_1}{\sqrt{2\gamma_2}},\ [Z_2,Y_2]_{\mathfrak{m}_{\{\alpha_1,\alpha_2,\alpha_3\}}}=\displaystyle -\frac{Z_4}{\sqrt{2\gamma_2}},\ [Z_2,Y_3]_{\mathfrak{m}_{\{\alpha_1,\alpha_2,\alpha_3\}}}=\displaystyle \frac{Z_3}{\sqrt{2\gamma_2}},\\
	\\
	&[Z_3,X_1]_{\mathfrak{m}_{\{\alpha_1,\alpha_2,\alpha_3\}}}=\displaystyle -\frac{Z_4}{\sqrt{2\gamma_1}},\ [Z_3,X_2]_{\mathfrak{m}_{\{\alpha_1,\alpha_2,\alpha_3\}}}=\displaystyle \frac{Z_1}{\sqrt{2\gamma_1}},\ [Z_3,X_3]_{\mathfrak{m}_{\{\alpha_1,\alpha_2,\alpha_3\}}}=\displaystyle \frac{Z_2}{\sqrt{2\gamma_1}},\\
	\\
	&[Z_3,Y_1]_{\mathfrak{m}_{\{\alpha_1,\alpha_2,\alpha_3\}}}=\displaystyle -\frac{Z_4}{\sqrt{2\gamma_2}},\ [Z_3,Y_2]_{\mathfrak{m}_{\{\alpha_1,\alpha_2,\alpha_3\}}}=\displaystyle \frac{Z_1}{\sqrt{2\gamma_2}},\ [Z_3,Y_3]_{\mathfrak{m}_{\{\alpha_1,\alpha_2,\alpha_3\}}}=\displaystyle -\frac{Z_2}{\sqrt{2\gamma_2}},\\
	\\
	&[Z_4,X_1]_{\mathfrak{m}_{\{\alpha_1,\alpha_2,\alpha_3\}}}=\displaystyle \frac{Z_3}{\sqrt{2\gamma_1}},\ [Z_4,X_2]_{\mathfrak{m}_{\{\alpha_1,\alpha_2,\alpha_3\}}}=\displaystyle -\frac{Z_2}{\sqrt{2\gamma_1}},\ [Z_4,X_3]_{\mathfrak{m}_{\{\alpha_1,\alpha_2,\alpha_3\}}}=\displaystyle \frac{Z_1}{\sqrt{2\gamma_1}},\\
	\\
	&[Z_4,Y_1]_{\mathfrak{m}_{\{\alpha_1,\alpha_2,\alpha_3\}}}=\displaystyle \frac{Z_3}{\sqrt{2\gamma_2}},\ [Z_4,Y_2]_{\mathfrak{m}_{\{\alpha_1,\alpha_2,\alpha_3\}}}=\displaystyle \frac{Z_2}{\sqrt{2\gamma_2}},\ [Z_4,Y_3]_{\mathfrak{m}_{\{\alpha_1,\alpha_2,\alpha_3\}}}=\displaystyle \frac{Z_1}{\sqrt{2\gamma_2}}.\\
	\end{align*}
	The Ricci components are given by 
	\begin{equation*}
	\begin{array}{lll}
	\Ric(Z_j,Z_j) & = & \displaystyle -\frac{3\gamma_1}{4\mu^2}-\frac{3\gamma_2}{4\mu^2}+\frac{6}{\mu},\ j=1,2,3,4,\\
	\\
	\Ric(X_j,X_j) & = & \displaystyle \frac{4}{\gamma_1}+\frac{\gamma_1}{2\mu^2},\ j=1,2,3,\\
	\\
	\Ric(Y_j,Y_j) & = &  \displaystyle\frac{4}{\gamma_2}+\frac{\gamma_2}{2\mu^2},\ j=1,2,3,
	\end{array}
	\end{equation*}
	therefore, if $A$ is an Einstein metric then
	\begin{align*}
	\displaystyle\frac{4}{\gamma_1}+\frac{\gamma_1}{2\mu^2}=\frac{4}{\gamma_2}+\frac{\gamma_2}{2\mu^2}&\Longleftrightarrow \displaystyle 8\mu^2\gamma_2+\gamma_1^2\gamma_2=8\mu^2\gamma_1+\gamma_1\gamma_2^2\\
	\\
	&\Longleftrightarrow 8\mu^2(\gamma_2-\gamma_1)-\gamma_1\gamma_2(\gamma_2-\gamma_1)=0\\
	\\
	&\Longleftrightarrow (\gamma_2-\gamma_1)(8\mu^2-\gamma_1\gamma_2)=0\\
	\\
	&\Longleftrightarrow \gamma_1=\gamma_2=:\gamma\ \text{or}\ \mu=\sqrt{\frac{\gamma_1\gamma_2}{8}}.
	\end{align*}
	We study the two cases:
	
\textit{Case 1. $\mu=\sqrt{\frac{\gamma_1\gamma_2}{8}}.$}
	
In this case $\Ric(Z_j,Z_j)=\displaystyle-\frac{3\gamma_1}{4\mu^2}-\frac{3\gamma_2}{4\mu^2}+\frac{6}{\mu}=\frac{12\sqrt{2}}{\sqrt{\gamma_1\gamma_2}}-\frac{6}{\gamma_1}-\frac{6}{\gamma_2},$ $\Ric(X_j,X_j)=\Ric(Y_j,Y_j)=\displaystyle\frac{4}{\gamma_1}+\frac{4}{\gamma_2}$ so
	\begin{align*}
	\displaystyle\frac{12\sqrt{2}}{\sqrt{\gamma_1\gamma_2}}-\frac{6}{\gamma_1}-\frac{6}{\gamma_2}=\frac{4}{\gamma_1}+\frac{4}{\gamma_2}&\Longleftrightarrow \displaystyle\frac{12\sqrt{2\gamma_1\gamma_2}-10\gamma_1-10\gamma_2}{\gamma_1\gamma_2}=0\\
	\\
	&\Longleftrightarrow 5\gamma_1-6\sqrt{2\gamma_1\gamma_2}+5\gamma_2=0\\
	\\
	&\Longleftrightarrow 5(\sqrt{\gamma_1}-\sqrt{\gamma_2})^2+(10-6\sqrt{2})\sqrt{\gamma_1\gamma_2}=0\\
	\\
	&\Longrightarrow \gamma_1=\gamma_2=0,
	\end{align*}
	which contradicts that $\gamma_1,\gamma_2>0.$ Hence, there is no Einstein metrics satisfying $\mu=\sqrt{\frac{\gamma_1\gamma_2}{8}}.$
	
\textit{Case 2. $\gamma_1=\gamma_2=:\gamma$.}
	
In this case $\Ric(Z_j,Z_j)=\displaystyle\frac{6}{\mu}-\frac{3\gamma}{2\mu^2}$ and $\Ric(X_j,X_j)=\Ric(Y_j,Y_j)=\displaystyle\frac{4}{\gamma}+\frac{\gamma}{2\mu^2}.$ Thus
	\begin{align*}
	\displaystyle\frac{6}{\mu}-\frac{3\gamma}{2\mu^2}=\frac{4}{\gamma}+\frac{\gamma}{2\mu^2}&\Longleftrightarrow \displaystyle 12\mu\gamma-3\gamma^2=8\mu^2+\gamma^2\\
	\\
	&\Longleftrightarrow \gamma^2-3\mu\gamma+2\mu^2=0\\
	\\
	&\Longleftrightarrow (\gamma-2\mu)(\gamma-\mu)=0\\
	\\
	&\Longleftrightarrow \displaystyle\mu=\frac{\gamma}{2}\ \text{or}\ \gamma=\mu,
	\end{align*}
	as we wanted to prove.
\end{proof}
\subsubsection{$(SO(l)\times SO(l+1))/(SO(d)\times SO(l-d)\times SO(l-d+1)),\ l\geq 3,\ 2\leq d\leq l-1$}

Let us consider $\mathfrak{g}=\mathfrak{so}(l+1,l)$ $l\geq 3$. Given $d\in\{2,...,l-1\}$, the flag manifold $\mathbb{F}_{\Sigma-\{\alpha_d\}}$ associated to $\Theta=\Sigma-\{\alpha_d\}$ is diffeomorphic to the homogeneous space
\begin{center}
	$(SO(l)\times SO(l+1))/(SO(d)\times SO(l-d)\times SO(l-d+1))$
\end{center}

In this case, the isotropy representation decomposes into the submodules $U_1=\vspan\{u_{st}:1\leq t<s\leq d\}$, $(V_1)_1=\vspan\{w_{st}-u_{st}:d+1\leq s\leq l,\ 1\leq t\leq d\}$ and $(V_1)_2=\vspan\{v_1,...,v_d\}\cup\{w_{st}+u_{st}:d+1\leq s\leq l,\ 1\leq t\leq d\},$ where the matrices $w_{ij},u_{ij},v_j$ are defined in \eqref{BasisB}. These subspaces are pairwise inequivalent, therefore every invariant metric $A$ with respect to the inner product \eqref{ProductB} has the form 
\begin{equation}
\label{metric10}
A|_{U_1}=\gamma I_{U_1},\ A|_{(V_1)_1}=\rho I_{(V_1)_1},\ A|_{(V_1)_2}=\mu I_{(V_1)_2}.
\end{equation}
\begin{pps}\label{p4.2.5} $a)$ If $d\neq 2,$  $\mathbb{F}_{\Sigma-\{\alpha_d\}}$ has at most four invariant Einstein metrics up to homotheties. In this case, a necessary condition for $\mathbb{F}_{\Sigma-\{\alpha_d\}}$ to have an Einstein metric is that the positive integers $l$, $d$ satisfy
	\begin{equation}\label{4.2.9}
	l^2(l-2)^2-2(d-1)^2(d-2)(2l-d)>0.
	\end{equation}
	$b)$ If $d=2$,  $\mathbb{F}_{\Sigma-\{\alpha_d\}}$ has at most three invariant Einstein metrics up to homotheties.
\end{pps}
\begin{proof} Let $A$ be an invariant metric on $\mathbb{F}_{\Sigma-\{\alpha_d\}}$ as in \eqref{metric10} and consider the $A-$orthonormal basis
	\begin{center}
		$Y_{st}=\displaystyle\frac{u_{st}}{\sqrt{\gamma}},\ 1\leq t<s\leq d,\ F_{st}=\frac{w_{st}-u_{st}}{\sqrt{2\rho}},\ 1\leq t\leq d,\ d+1\leq s\leq l,$ 
	\end{center}
	\begin{center}
		$Z_j=\displaystyle\frac{v_j}{\sqrt{\mu}},\ 1\leq j\leq d,\  G_{st}=\frac{w_{st}+u_{st}}{\sqrt{2\mu}},\ 1\leq t\leq d,\ d+1\leq s\leq l.$
	\end{center}
	The non-zero bracket relations of theses vectors are given by
	\begin{center}
		$\begin{array}{llll}
		[Y_{st},F_{is}]_{\mathfrak{m}_\Theta}=\displaystyle\frac{F_{it}}{\sqrt{\gamma}}, & [Y_{st},F_{it}]_{\mathfrak{m}_\Theta}=-\displaystyle\frac{F_{is}}{\sqrt{\gamma}}, & [Y_{st},G_{is}]_{\mathfrak{m}_\Theta}=-\displaystyle\frac{G_{it}}{\sqrt{\gamma}}, & [Y_{st},G_{it}]_{\mathfrak{m}_\Theta}=\displaystyle\frac{G_{is}}{\sqrt{\gamma}}\\
		&&&\\
		\left[Y_{st},Z_s\right]_{\mathfrak{m}_\Theta}=-\displaystyle\frac{Z_t}{\sqrt{\gamma}}, & [Y_{st},Z_t]_{\mathfrak{m}_\Theta}=\displaystyle\frac{Z_s}{\sqrt{\gamma}}, &[F_{st},F_{sj}]=\displaystyle\frac{\sqrt{\gamma}}{\rho}Y_{sj}, & [G_{st},G_{sj}]=\displaystyle\frac{\sqrt{\gamma}}{\mu}Y_{sj},\\
		\\
		\left[Z_s,Z_j\right]_{\mathfrak{m}_\Theta}=\displaystyle-\frac{\sqrt{\gamma}}{\mu}Y_{sj}, & \text{where}\ Y_{sj}=-Y_{js}\ \text{if} & s\leq j.
		\end{array}$
	\end{center}
	By \eqref{Scalar} we have that
	\begin{center}
		$S(A)=\displaystyle \frac{d(d-1)(d-2)}{\gamma}+d(l-d)\left(\frac{2(l-2)}{\rho}-\frac{(d-1)\gamma}{2\rho^2}\right)+d(l-d+1)\left(\frac{2(l-1)}{\mu}-\frac{(d-1)\gamma}{2\mu^2}\right).$
	\end{center}
	Suppose that $A$ has volume 1 and let $u=\frac{\rho}{\gamma},$ $v=\frac{\mu}{\gamma}$, then 
	\begin{align*}
	\displaystyle\frac{S(A)}{v_0^{\frac{2}{N}}}=\frac{S(u,v)}{v_0^{\frac{2}{N}}}=&\ d(d-1)(d-2)u^{\frac{d(l-d)}{N}}v^{\frac{d(l-d+1)}{N}}+2d(l-d)(l-2)u^{\frac{d(l-d)}{N}-1}v^{\frac{d(l-d+1)}{N}}\\
	\\
	&\ \displaystyle-\frac{d(l-d)(d-1)}{2}u^{\frac{d(l-d)}{N}-2}v^{\frac{d(l-d+1)}{N}}+2d(l-d+1)(l-1)u^{\frac{d(l-d)}{N}}v^{\frac{d(l-d+1)}{N}-1}\\
	\\
	&\ \displaystyle-\frac{d(l-d+1)(d-1)}{2}u^{\frac{d(l-d)}{N}}v^{\frac{d(l-d+1)}{N}-2},
	\end{align*}
	where $N=\frac{d(d-1)}{2}+d(l-d)+d(l-d+1)$ and $v_0=Vol(\mathbb{F}_\Theta,(\cdot,\cdot)|_{\mathfrak{m}_\Theta\times\mathfrak{m}_\Theta}).$ The partial derivatives of $S$ satisfy
	\begin{align*}
	\displaystyle f_1:=\frac{v_0^{-\frac{2}{N}}Nu^{-\frac{d(l-d)}{N}+3}v^{-\frac{d(l-d+1)}{N}+2}}{d^2(l-d)}\frac{\partial S}{\partial u}=&(d-1)(d-2)u^2v^2-2(l-d)(l-2)\left(\frac{N}{d(l-d)}-1\right)uv^2\\
	\\
	&+\frac{(l-d)(d-1)}{2}\left(\frac{2N}{d(l-d)}-1\right)v^2\\
	\\
	&\displaystyle+2(l-d+1)(l-1)u^2v-\frac{(l-d+1)(d-1)}{2}u^2,\\
	\\
	f_2:=\frac{v_0^{-\frac{2}{N}}Nu^{-\frac{d(l-d)}{N}+2}v^{-\frac{d(l-d+1)}{N}+3}}{d^2(l-d+1)}\frac{\partial S}{\partial v}=&(d-1)(d-2)u^2v^2+2(l-d)(l-2)uv^2\\
	\\
	&\displaystyle-\frac{(l-d)(d-1)}{2}v^2\\
	\\
	&\displaystyle-2(l-d+1)(l-1)\left(\frac{N}{d(l-d+1)}-1\right)u^2v\\
	\\
	&\displaystyle+\frac{(l-d+1)(d-1)}{2}\left(\frac{2N}{d(l-d+1)}-1\right)u^2.
	\end{align*}
	If $A$ is an Einstein metric, then $f_1=f_2=0$ and, therefore,
	\begin{equation*}
	\displaystyle\left(\frac{2N}{d(l-d+1)}-1\right)f_1+f_2=0\Longleftrightarrow C_1u^2v+C_2uv+C_3v+C_4u^2=0
	\end{equation*}
	where 
	\begin{center}
		$\displaystyle C_1=\frac{2N(d-1)(d-2)}{d(l-d+1)},\ C_2=-\frac{2N(l-2)l}{d(l-d+1)},\ C_3=\frac{N(d-1)(2l-d)}{d(l-d+1)},\ C_4=\frac{2N(l-1)}{d}.$
	\end{center}
	If $C_1u^2+C_2u+C_3=0$ then 
	\begin{center}
		$C_1u^2v+C_2uv+C_3v+C_4u^2=0\Longrightarrow C_4u^2=0\Longrightarrow u=0,$
	\end{center}
	a contradiction. Thus $C_1u^2+C_2u+C_3\neq 0$ and
	\begin{equation}\label{4.2.10}
	v=\displaystyle\frac{-C_4u^2}{C_1u^2+C_2u+C_3}.
	\end{equation}
	Suppose that $d\neq 2.$ Since $v>0$ and $-C_4u^2<0$, we have that $C_1u^2+C_2u+C_3<0$, so, the quadratic polynomial $C_1x^2+C_2x+C_3$ must have two distinct real roots (otherwise, it would always  be non-negative since $C_1>0$), but this occurs when
	\begin{align*}
	C_2^2-4C_1C_3>0&\Longleftrightarrow \displaystyle\frac{4N^2(l-2)^2l^2}{d^2(l-d+1)^2}-\frac{8N^2(d-1)^2(d-2)(2l-d)}{d^2(l-d+1)^2}>0\\
	\\
	&\Longleftrightarrow l^2(l-2)^2-2(d-1)^2(d-2)(2l-d)>0.
	\end{align*}
	Now, substituting \eqref{4.2.10} in $f_1=0$ and multiplying it by $(C_1u^2+C_2u+C_3)^2$ we obtain
	\begin{align*}
	0=(C_1u^2+C_2u+C_3)^2f_1=&\displaystyle(d-1)(d-2)C_4^2u^4-2(l-d)(l-2)\left(\frac{N}{d(l-d)}-1\right)C_4^2u^3\\
	\\
	&\displaystyle+\frac{(l-d)(d-1)}{2}\left(\frac{2N}{d(l-d)}-1\right)C_4^2u^2\\
	\\
	&\displaystyle+2(l-d+1)(l-1)C_4u^2(C_1u^2+C_2u+C_3)-\frac{(l-d+1)(d-1)}{2}\\
	\\
	&=:g(u),
	\end{align*}
	where $g$ is a fourth-degree polynomial. Hence, the set of invariant Einstein metrics $(1,u,v)$ of volume 1 is contained in the set \begin{equation*}
	\{1\}\times\left\{(u,v)\in(\mathbb{R}^+)^2:g(u)=0\ \text{and}\ v=\displaystyle\frac{-C_4u^2}{C_1u^2+C_2u+C_3}\right\}
	\end{equation*}
	which has at most four elements. This concludes the proof of item $a).$ Now, when $d=2$. Then, item $b)$ follows from the fact that $C_1=0$ and, therefore, $g$ is third-degree polynomial.
\end{proof}
\subsubsection{$U(l)/(O(d)\times U(l-d)),\ l\geq 3,\ 2\leq d\leq l-1$}
Here we are considering $\mathfrak{g}=\mathfrak{sp}(l,\mathbb{R}),$ $l\geq 3$ and $\Theta=\Sigma-\{\alpha_d\},$ where $2\leq d\leq l-1$. The associated flag manifold  is diffeomorphic to $U(l)/(O(d)\times U(l-d))$. Fix the $\Ad(U(l))-$invariant product \eqref{ProductC} and the matrices \eqref{BasisC}. The subspace $\mathfrak{m}_\Theta$ is decomposed into the inequivalent submodules $M_{21}=\vspan\{w_{st},u_{st}:1\leq t\leq d,\ d+1\leq s\leq l\},$ $U_1=\vspan\{u_{jj}-u_{j+1,j+1}:j=1,...,d-1\}\cup\{u_{st}:1\leq t<s\leq d\},$ $V_1=\vspan\{u_{11}+...+u_{dd}\}.$ Every invariant metric $A$ is given by positive numbers $\mu_0,\mu_1,\mu_{21}$ such that 
\begin{equation}\label{metric11}
A|_{V_1}=\mu_0I_{V_1},\ A|_{U_1}=\mu_1I_{U_1},\ A|_{M_{21}}=\mu_{21}I_{M_{21}}
\end{equation}
\begin{pps}
	The flag $U(l)/(O(d)\times U(l-d)),$ $l\geq 3,$ and $2\leq d\leq l-1,$ has at most two invariant Einstein metrics up to homotheties.
\end{pps}
\begin{proof} Let $A$ be an invariant metric on $U(l)/(O(d)\times U(l-d))$ as in \eqref{metric11}. Consider the $A-$orthonormal basis of $\mathfrak{m}_\Theta$ given by
	\begin{equation*}
		Z_1=\displaystyle\sqrt{\frac{2}{d\mu_0}}\left(u_{11}+...+u_{dd}\right),\ T_j=\sqrt{\frac{2j}{(j+1)\mu_1}}\left(\frac{1}{j}(u_{11}+...+u_{jj})-u_{j+1,j+1}\right),
	\end{equation*}
	\begin{equation*}
		Y_{st}=\displaystyle \frac{u_{st}}{\sqrt{\mu_1}},\ 1\leq t<s\leq d,\ X_{st}=\frac{w_{st}}{\sqrt{\mu_{21}}},\ Y_{st}=\frac{u_{st}}{\sqrt{\mu_{21}}},\ 1\leq t\leq d,\ d+1\leq s\leq l.
	\end{equation*}
	Then, we have the following bracket relations:
	\begin{align*}
	&[Z_1,X_{st}]_{\mathfrak{m}_\Theta}=\displaystyle-\sqrt{\frac{2}{d\mu_0}}Y_{st},\ [Z_1,Y_{st}]_{\mathfrak{m}_\Theta}=\displaystyle\sqrt{\frac{2}{d\mu_0}}X_{st},\ 1\leq t\leq d< s\leq l,\\
	\\
	&[T_j,X_{st}]_{\mathfrak{m}_\Theta}=\displaystyle-\sqrt{\frac{2}{(j+1)j\mu_1}}Y_{st},\ [T_{t-1},X_{st}]_{\mathfrak{m}_\Theta}=\displaystyle\sqrt{\frac{2(t-1)}{t\mu_1}}Y_{st},\ 1\leq t\leq j\leq d< s\leq l,\\
	\\
	&[T_j,Y_{st}]_{\mathfrak{m}_\Theta}=\displaystyle\sqrt{\frac{2}{(j+1)j\mu_1}}X_{st},\ [T_{t-1},Y_{st}]_{\mathfrak{m}_\Theta}=\displaystyle-\sqrt{\frac{2(t-1)}{t\mu_1}}X_{st},\ 1\leq t\leq j\leq d< s\leq l,\\
	\\
	&[Y_{ij},X_{si}]_{\mathfrak{m}_\Theta}=\displaystyle-\frac{Y_{sj}}{\sqrt{\mu_1}},\ [Y_{ij},X_{sj}]_{\mathfrak{m}_\Theta}=\displaystyle-\frac{Y_{si}}{\sqrt{\mu_1}},\ 1\leq j<i\leq d<s\leq l,\\
	\\
	&[Y_{ij},Y_{si}]_{\mathfrak{m}_\Theta}=\displaystyle\frac{X_{sj}}{\sqrt{\mu_1}},\ [Y_{ij},Y_{sj}]_{\mathfrak{m}_\Theta}=\displaystyle\frac{X_{si}}{\sqrt{\mu_1}},\ 1\leq j<i\leq d<s\leq l,\\
	\\
	&[X_{sj},Y_{st}]_{\mathfrak{m}_\Theta}=\displaystyle-\sqrt{\frac{\mu_1}{\mu^2_{21}}}Y_{jt},\ 1\leq j\neq t\leq d<s\leq l\hspace{1cm} \text{(}Y_{jt}=-Y_{tj}\ \text{if}\ j<t\text{),}\\
	\\
	&[X_{st},Y_{st}]_{\mathfrak{m}_\Theta}=\displaystyle-\sqrt{\frac{2\mu_0}{d\mu_{21}^2}}Z_1+\sqrt{\frac{2(t-1)\mu_1}{t\mu_{21}^2}}T_{t-1}-\sum\limits_{j=t}^{d-1}\sqrt{\frac{2\mu_1}{j(j+1)\mu_{21}^2}}T_j,\ 1\leq t\leq d<s\leq l,\\
	\end{align*}
	where $T_0=0$ and $\sum\limits_{j=t}^{d-1}\sqrt{\frac{2\mu_1}{j(j+1)\mu_{21}^2}}T_j=0$ if $t=d.$ Computing the scalar curvature we obtain 
	\begin{equation*}
	S(A)=\displaystyle\frac{d(d-1)(d+2)}{\mu_1}-\frac{(l-d)\mu_0}{\mu_{21}^2}-\frac{(l-d)(d-1)(d+2)\mu_1}{2\mu_{21}^2}+\frac{4ld(l-d)}{\mu_{21}}.
	\end{equation*}
	Assume that $A$ has volume 1, that is,
	\begin{equation*}
	\mu_0\mu_1^{\frac{(d-1)(d+2)}{2}}\mu_{21}^{2d(l-d)}=\frac{1}{v_0^2},
	\end{equation*}
	where $v_0=Vol(U(l)/(O(d)\times U(l-d)),(\cdot,\cdot)|_{\mathfrak{m}_\Theta\times\mathfrak{m}_\Theta}).$ Let $u=\frac{\mu_1}{\mu_0}$, $v=\frac{\mu_{21}}{\mu_0}$, then 
	\begin{align*}
	v_0^{-\frac{2}{N}}S(A)=v_0^{-\frac{2}{N}}S(u,v)=&\ \displaystyle d(d-1)(d+2)u^{\frac{(d-1)(d+2)}{2N}-1}v^{\frac{2d(l-d)}{N}}-(l-d)u^{\frac{(d-1)(d+2)}{2N}}v^{\frac{2d(l-d)}{N}-2}\\
	\\
	&\displaystyle\ -\frac{(l-d)(d-1)(d+2)}{2}u^{\frac{(d-1)(d+2)}{2N}+1}v^{\frac{2d(l-d)}{N}-2}\\
	\\
	&\displaystyle\ +4ld(l-d)u^{\frac{(d-1)(d+2)}{2N}}v^{\frac{2d(l-d)}{N}-1},
	\end{align*} 
	where $N=\frac{(d-1)(d+2)}{2}+2d(l-d)+1$ is the dimension of the flag. Therefore,
	\begin{align*}
	g_1:=\displaystyle\frac{2Nu^{-\frac{(d-1)(d+2)}{2N}+2}v^{-\frac{2d(l-d)}{N}+2}v_0^{-\frac{2}{N}}}{(d-1)(d+2)}\frac{\partial S}{\partial u}=&-d(d-1)(d+2)\left(\frac{2N}{(d-1)(d+2)}-1\right)v^2-(l-d)u\\
	\\
	&-\frac{(l-d)(d-1)(d+2)}{2}\left(\frac{2N}{(d-1)(d+2)}+1\right)u^2\\
	\\
	&+4ld(l-d)uv,\\
	\\
	g_2:=\displaystyle\frac{Nu^{-\frac{(d-1)(d+2)}{2N}+1}v^{-\frac{2d(l-d)}{N}+3}v_0^{-\frac{2}{N}}}{2d(l-d)}\frac{\partial S}{\partial v}=&\ d(d-1)(d+2)v^2+(l-d)\left(\frac{N}{d(l-d)}-1\right)u\\
	\\
	&\ \displaystyle+\frac{(l-d)(d-1)(d+2)}{2}\left(\frac{N}{d(l-d)}-1\right)u^2\\
	\\
	&\ \displaystyle-4ld(l-d)\left(\frac{N}{2d(l-d)}-1\right)uv.
	\end{align*}
	By Proposition \ref{Propositionscalar}, if $A$ is an Einstein metric then $g_1=g_2=0$, thus 
	\begin{equation*}
	g_1+\left(\frac{2N}{(d-1)(d+2)}-1\right)g_2=0\Longleftrightarrow C_1u+C_2+C_3v=0,
	\end{equation*}
	where
	\begin{center}
		$C_1=\displaystyle\frac{N}{d},\ C_2=\displaystyle\frac{2N(d(l-d)+1)}{d(d-1)(d+2)},\ C_3=-\frac{4lN}{(d-1)(d+2)},$
	\end{center}
	so we have that 
	\begin{equation}\label{4.2.12}
	v=D_1u+D_2,
	\end{equation}
	where $D_1=-\frac{C_1}{C_3}$ and $D_2=-\frac{C_2}{C_3}.$ Substituting \eqref{4.2.12} in $g_1=0$ we obtain
	\begin{align*}
	0=g_1=&\ -d(d-1)(d+2)\left(\frac{2N}{(d-1)(d+2)}-1\right)(D_1u+D_2)^2-(l-d)u\\
	\\
	&\ -\frac{(l-d)(d-1)(d+2)}{2}\left(\frac{2N}{(d-1)(d+2)}+1\right)u^2+4ld(l-d)u(D_1u+D_2)\\
	\\
	=:&\ h(u),
	\end{align*}
	where $h$ is a two-degree polynomial. The result follows from the fact that $h$ has at most two positive roots and formula \eqref{4.2.12}.
\end{proof}
\subsubsection{$(SO(l)\times SO(l))/S(O(l-1)\times O(1))$}\label{ss4.2.6}
Let us consider $\mathfrak{g}=\mathfrak{so}(l,l),\ l\geq 4,$ $\Theta=\{\alpha_1,...,\alpha_{l-2}\},$ $(\cdot,\cdot)$ as in \eqref{ProductD} and the $(\cdot,\cdot)-$orthonormal basis \eqref{BasisD}. In this case, $K$ is diffeomorphic to $SO(l)\times SO(l)$ and $K_\Theta$ is diffemorphic to $S(O(l-1)\times O(1)).$ The isotropy representation of $K_\Theta$ on $\mathfrak{m}_\Theta$ decomposes into the irreducible submodule $U_1=\vspan\{u_{ij}:1\leq t<s\leq l-1\},$ which is not equivalent to any other submodule, and the equivalent irreducible sumodules $W_{21}=\vspan\{w_{lj}:1\leq j\leq l-1\}$ and $U_{21}=\vspan\{u_{lj}:1\leq j\leq l-1\}.$ Analogously to the proof of Proposition \ref{non-diagonalA}, we can show that every invariant metric $A$ has the form
\begin{equation}\label{metric12}
A\left|_{U_1}\right.=\gamma I_{U_1},\ Aw_{lj}=\lambda_1w_{lj}+bu_{lj},\ Au_{lj}=bw_{lj}+\lambda_2u_{lj},\ j=1,...,l-1.
\end{equation}
for some $\mu,\lambda_1,\lambda_2>0$ and $b\in\mathbb{R}.$
\begin{pps}\label{p4.13}
	Let $A$ be an invariant metric on the flag $\mathbb{F}_{\{\alpha_1,...,\alpha_{l-2}\}}$ written as in \eqref{metric12}. Then $A$ is an Einstein metric if and only if $A$ satisfies one of the following conditions:
	\begin{itemize}
		\item[(F1)] $b=0,$ $\lambda_1=\left(1-\frac{\sqrt{l^2-5l+4}}{2(l-1)}\right)\gamma$ \ and \ $\lambda_2=\left(1+\frac{\sqrt{l^2-5l+4}}{2(l-1)}\right)\gamma$
		\item[(F2)] $b=0,$ $\lambda_1=\left(1+\frac{\sqrt{l^2-5l+4}}{2(l-1)}\right)\gamma$ \ and \ $\lambda_2=\left(1-\frac{\sqrt{l^2-5l+4}}{2(l-1)}\right)\gamma$
		\item[(F3)] $b>0,$ $b=\lambda_1=\frac{\lambda_2}{3}=\frac{\gamma}{2}$
		\item[(F4)] $b>0,$ $b=\lambda_2=\frac{\lambda_1}{3}=\frac{\gamma}{2}$
		\item[(F5)] $b<0,$ $-b=\lambda_1=\frac{\lambda_2}{3}=\frac{\gamma}{2}$
		\item[(F6)] $b<0,$ $-b=\lambda_2=\frac{\lambda_1}{3}=\frac{\gamma}{2}$
	\end{itemize}
\end{pps}
\begin{proof}
	The proof is analogous to the proof of Proposition \ref{p4.7}. As before, we consider two cases:
	
$\bullet$ $A$ diagonal ($b=0$):
	
In this case, we have the $A-$orthonormal basis of $\mathfrak{m}_\Theta$ given by the vectors
	\begin{equation}\label{D2}
	\displaystyle Y_{ij}=\frac{u_{ij}}{\sqrt{\gamma}},\ 1\leq j<i\leq l-1,\  X_{lj}=\frac{w_{lj}}{\sqrt{\lambda_1}},\ Y_{lj}=\frac{u_{lj}}{\sqrt{\lambda_2}},\ j=1,...,l-1,
	\end{equation}
	which satisfy the following bracket relations
	\begin{equation*}
	\begin{array}{lclclcl}
	\displaystyle[Y_{ij},X_{li}]_{\mathfrak{m}_\Theta} & = & \displaystyle-\sqrt{\frac{\lambda_2}{\gamma\lambda_1}}Y_{lj},& & \displaystyle[Y_{ij},X_{lj}]_{\mathfrak{m}_\Theta} & = & \displaystyle\sqrt{\frac{\lambda_2}{\gamma\lambda_1}}Y_{li},\\
	\\
	\displaystyle[Y_{ij},Y_{li}]_{\mathfrak{m}_\Theta} & = & \displaystyle-\sqrt{\frac{\lambda_1}{\gamma\lambda_2}}X_{lj},& & \displaystyle[Y_{ij},Y_{lj}]_{\mathfrak{m}_\Theta} & = & \displaystyle\sqrt{\frac{\lambda_1}{\gamma\lambda_2}}X_{li},\\
	\\
	\displaystyle[X_{lt},Y_{ls}]_{\mathfrak{m}_\Theta} & = & \displaystyle\sqrt{\frac{\gamma}{\lambda_1\lambda_2}}Y_{st},& &\text{ for }s\neq t, &&\\
	\end{array}
	\end{equation*}
	where, $Y_{st}=-Y_{ts}$ if $s<t.$ Observe that $[X,Y]$ is $A-$orthogonal to $X$ and $Y$, for all $X,Y$ in the basis \eqref{D2}, thus, $Z=\displaystyle\sum U(X,X)=0,$ where the sum extends over the basis \eqref{D2}. By formula \eqref{ric}, we obtain that
	\begin{equation*}
	\begin{array}{rllll}
	r_0 & := & \Ric(Y_{ij},Y_{ij}) & = & \displaystyle \frac{2(l-2)}{\gamma}+\frac{\gamma}{\lambda_1\lambda_2}-\frac{\lambda_1}{\gamma\lambda_2}-\frac{\lambda_2}{\gamma\lambda_1},\ 1\leq j<i\leq l-1,\\
	\\
	r_1 & := & \Ric(X_{lj},X_{lj}) & = & \displaystyle(l-2)\left( \frac{2}{\lambda_1}+\frac{\lambda_1}{2\gamma\lambda_2}-\frac{\lambda_2}{2\gamma\lambda_1}-\frac{\gamma}{2\lambda_1\lambda_2}\right),\ 1\leq j\leq l-1,\\
	\\
	r_2 & := & \Ric(Y_{lj},Y_{lj}) & = & \displaystyle(l-2)\left( \frac{2}{\lambda_2}+\frac{\lambda_2}{2\gamma\lambda_1}-\frac{\lambda_1}{2\gamma\lambda_2}-\frac{\gamma}{2\lambda_1\lambda_2}\right),\ 1\leq j\leq l-1.
	\end{array}
	\end{equation*}
	Therefore, $A$ is an Einstein metric if and only if $r_0=r_1=r_2$, i.e.,
	
	\begin{center}
		$\left\{\begin{array}{l}
		\displaystyle\lambda_1=\left(1-\frac{\sqrt{l^2-5l+4}}{2(l-1)}\right)\gamma\\
		\\
		\displaystyle\lambda_2=\left(1+\frac{\sqrt{l^2-5l+4}}{2(l-1)}\right)\gamma\\	
		\end{array}\right.$\ \ \ \ \ \ or\ \ \ \ \ \ $\left\{\begin{array}{l}
		\displaystyle\lambda_1=\left(1+\frac{\sqrt{l^2-5l+4}}{2(l-1)}\right)\gamma\\
		\\
		\displaystyle\lambda_2=\left(1-\frac{\sqrt{l^2-5l+4}}{2(l-1)}\right)\gamma\\
		\end{array}.\right.$ 
	\end{center}
	$\bullet$ $A$ non-diagonal ($b\neq 0$):
	
The eigenvalues of $A$ are given by
	\begin{center}
		$\xi_0=\gamma,$\  $\displaystyle\xi_1=\frac{1}{2}\left(\lambda_1+\lambda_2-\sqrt{4b^2+(\lambda_2-\lambda_1)^2}\right)$ and $\displaystyle\xi_2=\frac{1}{2}\left(\lambda_1+\lambda_2+\sqrt{4b^2+(\lambda_2-\lambda_1)^2}\right)$
	\end{center}
	and satisfy
	\begin{equation*}
	\begin{array}{lcl}
	&&\xi_1+\xi_2=\lambda_1+\lambda_2, \ \ \ \xi_1\xi_2=\lambda_1\lambda_2-b^2,\\
	\\
	b^2 & = & (\xi_2-\lambda_2)(\xi_2-\lambda_1)=(\lambda_1-\xi_1)(\xi_2-\lambda_1)\\
	\\
	& = & (\lambda_2-\xi_1)(\lambda_1-\xi_1)=(\lambda_2-\xi_1)(\xi_2-\lambda_2).
	\end{array}
	\end{equation*}
	Since $A$ is a positive operator, then $\xi_1,\xi_2>0.$ An $A-$orthonormal basis of $\mathfrak{m}_\Theta$ is given by
	\begin{equation*}
		\begin{array}{cl}
		\displaystyle Y_{ij}=\frac{u_{ij}}{\sqrt{\xi_0}},\ 1\leq j<i\leq l-1, & \displaystyle X_{lj}=\frac{(\xi_1-\lambda_2)w_{lj}+bu_{lj}}{\sqrt{c_1}},\ Y_{lj}=\frac{(\xi_2-\lambda_2)w_{lj}+bu_{lj}}{\sqrt{c_2}},\ 1\leq j\leq l-1,\\ 
		\end{array}
	\end{equation*}
	where $c_1=\xi_1(\xi_1-\lambda_2)(\xi_1-\xi_2)$ and $c_2=\xi_2(\xi_2-\lambda_2)(\xi_2-\xi_1).$ These vectors satisfy the relations
	\begin{center}
		$[X_{lt},X_{ls}]_{\mathfrak{m}_\Theta}=\frac{2b\sqrt{\xi_0}}{\xi_1(\xi_1-\xi_2)}Y_{st},\ \ [X_{lt},Y_{ls}]_{\mathfrak{m}_\Theta}=\frac{b(\lambda_2-\lambda_1)}{|b|(\xi_1-\xi_2)}\sqrt{\frac{\xi_0}{\xi_1\xi_2}}Y_{st},\ \ [Y_{lt},Y_{ls}]_{\mathfrak{m}_\Theta}=\frac{2b\sqrt{\xi_0}}{\xi_2(\xi_2-\xi_1)}Y_{st},\ s\neq t,$
	\end{center}
	\begin{align*}
	&[Y_{ij},X_{lj}]_{\mathfrak{m}_\Theta}=\frac{b}{\xi_1-\xi_2}\left(\frac{2}{\sqrt{\xi_0}}X_{li}+\sqrt{\frac{\xi_2}{\xi_0\xi_1}}\left(\frac{\lambda_2-\lambda_1}{|b|}\right)Y_{li}\right),\\
	\\
	&[Y_{ij},Y_{li}]_{\mathfrak{m}_\Theta}=\frac{b}{\xi_1-\xi_2}\left(\frac{2}{\sqrt{\xi_0}}Y_{lj}+\sqrt{\frac{\xi_1}{\xi_0\xi_2}}\left(\frac{\lambda_1-\lambda_2}{|b|}\right)X_{lj}\right),\\
	\\
	&[Y_{ij},Y_{lj}]_{\mathfrak{m}_\Theta}=\frac{b}{\xi_2-\xi_1}\left(\frac{2}{\sqrt{\xi_0}}Y_{lj}+\sqrt{\frac{\xi_1}{\xi_0\xi_2}}\left(\frac{\lambda_1-\lambda_2}{|b|}\right)X_{lj}\right),\\
	\\
	&[Y_{ij},X_{li}]_{\mathfrak{m}_\Theta}=\frac{b}{\xi_2-\xi_1}\left(\frac{2}{\sqrt{\xi_0}}X_{lj}+\sqrt{\frac{\xi_2}{\xi_0\xi_1}}\left(\frac{\lambda_2-\lambda_1}{|b|}\right)Y_{lj}\right),\ 1\leq j<i\leq l-1.\\
	\end{align*}
	We may use formula \eqref{ric} to obtain
	\begin{equation*}
		\displaystyle \Ric(X_{lj},Y_{lj})=\frac{(l-2)|b|(\lambda_1-\lambda_2)(\xi_0^2-2\xi_1\xi_2)}{(\xi_2-\xi_1)\xi_0(\xi_1\xi_2)^{\frac{3}{2}}},\ \text{for all}\ j\in\{1,...,l-1\}.
	\end{equation*}
	Thus, if $A$ is an Einstein metric then $\Ric(X_{lj},Y_{lj})=cg(X_{lj},Y_{lj})=0$ (for some $c\in\mathbb{R}$), i.e., $\lambda_1=\lambda_2$ or $\xi_0=\sqrt{2\xi_1\xi_2}.$
	
\textit{Case 1. $\lambda_1=\lambda_2.$}
	
When $\lambda_1=\lambda_2$, the Ricci components are guven by
	\begin{equation*}
	\begin{array}{rllll}
	r_0 & := & \Ric(Y_{ij},Y_{ij}) & = & \displaystyle \frac{2(l-3)}{\xi_0}+\frac{\xi_0}{2}\left(\frac{1}{\xi_1^2}+\frac{1}{\xi_2^2}\right),\ 1\leq j<i\leq l-1,\\
	\\
	r_1 & := & \Ric(X_{lj},X_{lj}) & = & \displaystyle(l-2)\left( \frac{2}{\xi_1}-\frac{\xi_0}{2\xi_1^2}\right),\ 1\leq j\leq l-1,\\
	\\
	r_2 & := & \Ric(Y_{lj},Y_{lj}) & = & \displaystyle(l-2)\left( \frac{2}{\xi_2}-\frac{\xi_0}{2\xi_2^2}\right),\ 1\leq j\leq l-1.
	\end{array}
	\end{equation*}
	We shall show that the sytem of equations
	\begin{equation}\label{D3}
	\left\{\begin{array}{l}
	r_0=r_1  \\
	\\
	r_1=r_2
	\end{array}\right.
	\end{equation}
	has not positive solutions. In fact,
	\begin{align*}
	r_1=r_2 &\Longleftrightarrow \displaystyle \frac{2}{\xi_1}-\frac{\xi_0}{2\xi_1^2}=\frac{2}{\xi_2}-\frac{\xi_0}{2\xi_2^2}\\
	\\
	&\Longleftrightarrow 4\xi_1\xi_2^2-\xi_0\xi_2^2=4\xi_1^2\xi_2-\xi_0\xi_1^2\\
	\\
	&\Longleftrightarrow 4\xi_1\xi_2(\xi_2-\xi_1)-\xi_0(\xi_2-\xi_1)(\xi_2+\xi_1)=0\\
	\\
	&\Longleftrightarrow (\xi_2-\xi_1)(4\xi_1\xi_2-\xi_0(\xi_1+\xi_2))=0\\
	\\
	&\Longleftrightarrow \xi_1=\xi_2\ \text{or}\ \xi_0=\displaystyle\frac{4\xi_1\xi_2}{\xi_1+\xi_2}.
	\end{align*}
	Since $b\neq 0$ then $\xi_1\neq\xi_2.$ Assuming $\xi_0=\displaystyle\frac{4\xi_1\xi_2}{\xi_1+\xi_2}$ we have that
	\begin{equation*}
	r_0=\displaystyle\frac{(l+1)\xi_1^2+(l+1)\xi_2^2+2(l-3)\xi_1\xi_2}{2\xi_1\xi_2(\xi_1+\xi_2)}\ \text{and}\ r_1=r_2=\displaystyle\frac{4(l-2)\xi_1\xi_2}{2\xi_1\xi_2(\xi_1+\xi_2)},
	\end{equation*}
	therefore
	\begin{align*}
	r_0-r_1=0 &\Longleftrightarrow \displaystyle\frac{(l+1)\xi_1^2+(l+1)\xi_2^2-2(l-1)\xi_1\xi_2}{2\xi_1\xi_2(\xi_1+\xi_2)}=0\\
	\\
	&\Longleftrightarrow (l+1)(\xi_1-\xi_2)^2+4\xi_1\xi_2=0\\
	\\
	&\Longrightarrow \xi_1=\xi_2=0.
	\end{align*}
	Hence, system \eqref{D3} has no solutions $\xi_1,\xi_2$ with $\xi_1,\xi_2>0.$
	
\textit{Case 2. $\xi_0=\sqrt{2\xi_1\xi_2}.$}
	
In this case we have
	\begin{equation*}
	\begin{array}{rllll}
	r_0 & := & \Ric(Y_{ij},Y_{ij}) & = & \displaystyle \frac{2(l-2)}{\xi_0}+\frac{8b^2-(\xi_2-\xi_1)^2}{\xi_0\xi_1\xi_2},\\
	\\
	r_1 & := & \Ric(X_{lj},X_{lj}) & = & \displaystyle(l-2)\left( \frac{2}{\xi_1}-\frac{2b^2\xi_0}{\xi_1^2(\xi_2-\xi_1)^2}+\frac{(\lambda_2-\lambda_1)^2}{2(\xi_2-\xi_1)^2}\left(\frac{\xi_1^2-\xi_2^2-\xi_0^2}{\xi_0\xi_1\xi_2}\right)\right),\\
	\\
	r_2 & := & \Ric(Y_{lj},Y_{lj}) & = & \displaystyle(l-2)\left( \frac{2}{\xi_2}-\frac{2b^2\xi_0}{\xi_2^2(\xi_2-\xi_1)^2}+\frac{(\lambda_2-\lambda_1)^2}{2(\xi_2-\xi_1)^2}\left(\frac{\xi_2^2-\xi_1^2-\xi_0^2}{\xi_0\xi_1\xi_2}\right)\right).
	\end{array}
	\end{equation*}
	Thus
	\begin{align*}
	r_1-r_2=&\  (l-2)\left(2\left(\frac{1}{\xi_1}-\frac{1}{\xi_2}\right)-\frac{2b^2\xi_0}{(\xi_2-\xi_1)^2}\left(\frac{1}{\xi_1^2}-\frac{1}{\xi_2^2}\right)+\frac{(\lambda_2-\lambda_1)^2}{\xi_0(\xi_2-\xi_1)^2}\left(\frac{\xi_1}{\xi_2}-\frac{\xi_2}{\xi_1}\right)\right)\\
	\\
	=&\  (l-2)\left(2\left(\frac{\xi_2-\xi_1}{\xi_1\xi_2}\right)-\frac{2b^2\xi_0}{(\xi_2-\xi_1)^2}\left(\frac{\xi_2^2-\xi_1^2}{\xi_1^2\xi_2^2}\right)+\frac{(\lambda_2-\lambda_1)^2}{(\xi_2-\xi_1)^2}\left(\frac{\xi_1^2-\xi_2^2}{\xi_0\xi_1\xi_2}\right)\right)\\
	\\
	=&\  (l-2)\left(2\left(\frac{\xi_2-\xi_1}{\xi_1\xi_2}\right)-\frac{4b^2}{(\xi_2-\xi_1)^2}\left(\frac{\xi_2^2-\xi_1^2}{\xi_0\xi_1\xi_2}\right)+\frac{(\lambda_2-\lambda_1)^2}{(\xi_2-\xi_1)^2}\left(\frac{\xi_1^2-\xi_2^2}{\xi_0\xi_1\xi_2}\right)\right)\\
	\\
	=&\  (l-2)\left(2\left(\frac{\xi_2-\xi_1}{\xi_1\xi_2}\right)-\frac{\xi_2^2-\xi_1^2}{(\xi_2-\xi_1)^2}\left(\frac{4b^2+(\lambda_2-\lambda_1)^2}{\xi_0\xi_1\xi_2}\right)\right)\\
	\\
	=&\  (l-2)\left(2\left(\frac{\xi_2-\xi_1}{\xi_1\xi_2}\right)-\frac{\xi_2^2-\xi_1^2}{\xi_0\xi_1\xi_2}\right)\\
	\\
	=&\  \frac{(l-2)(\xi_2-\xi_1)(2\xi_0-(\xi_1+\xi_2))}{\xi_0\xi_1\xi_2}\\
	\\
	=&\  \frac{(l-2)(\xi_2-\xi_1)(2\sqrt{2\xi_1\xi_2}-(\xi_1+\xi_2))}{\sqrt{2}(\xi_1\xi_2)^{\frac{3}{2}}}.
	\end{align*}
	Observe that $\xi_2-\xi_1\neq 0$ (since $b\neq0$), therefore, an Einstein metric $A$ satisfiying $b\neq0$ and $\xi_0=\sqrt{2\xi_1\xi_2}$ also satisfies $2\sqrt{2\xi_1\xi_2}=\xi_1+\xi_2.$ In this situation we have
	\begin{align*}
	r_0-\frac{r_1+r_2}{2}=&\ \frac{2(l-2)}{\xi_0}+\frac{8b^2-(\xi_2-\xi_1)^2}{\xi_0\xi_1\xi_2}\\
	\\
	&-\left(\frac{l-2}{2}\right)\left(2\left(\frac{1}{\xi_1}+\frac{1}{\xi_2}\right)-\frac{2b^2\xi_0}{(\xi_2-\xi_1)^2}\left(\frac{1}{\xi_1^2}+\frac{1}{\xi_2^2}\right)-\frac{(\lambda_2-\lambda_1)^2}{(\xi_2-\xi_1)^2}\left(\frac{\xi_0}{\xi_1\xi_2}\right)\right)\\
	\\
	=&\ 2(l-2)\left(\frac{1}{\xi_0}-\frac{1}{2\xi_1}-\frac{1}{2\xi_2}\right)+\frac{8b^2-(\xi_2-\xi_1)^2}{\xi_0\xi_1\xi_2}+\frac{2b^2(l-2)}{(\xi_2-\xi_1)^2}\left(\frac{\xi_1^2+\xi_2^2}{\xi_0\xi_1\xi_2}\right)\\
	\\
	&+\frac{(\lambda_2-\lambda_1)^2(l-2)}{(\xi_2-\xi_1)^2}\left(\frac{\xi_0}{2\xi_1\xi_2}\right)\\
	\\
	=&\ \frac{(l-2)(2\xi_1\xi_2-\xi_0(\xi_1+\xi_2))}{\xi_0\xi_1\xi_2}+\frac{8b^2-(\xi_2-\xi_1)^2}{\xi_0\xi_1\xi_2}\\
	\\
	&+\left(\frac{l-2}{(\xi_2-\xi_1)^2}\right)\left(\frac{2b^2(\xi_1^2+\xi_2^2)+(\lambda_2-\lambda_1)^2\xi_1\xi_2}{\xi_0\xi_1\xi_2}\right)\\
	\\
	=&\ \frac{(l-2)(2\xi_1\xi_2-2\xi_0^2)}{\xi_0\xi_1\xi_2}+\frac{8b^2-(\xi_2-\xi_1)^2}{\xi_0\xi_1\xi_2}\\
	\\
	&+\left(\frac{l-2}{(\xi_2-\xi_1)^2}\right)\left(\frac{2b^2(\xi_1^2+\xi_2^2)+(\xi_2-\xi_1)^2\xi_1\xi_2-4b^2\xi_1\xi_2}{\xi_0\xi_1\xi_2}\right)\\
	\\
	=&\ \frac{(l-2)(-2\xi_1\xi_2)}{\xi_0\xi_1\xi_2}+\frac{8b^2-(\xi_2-\xi_1)^2}{\xi_0\xi_1\xi_2}+\left(\frac{l-2}{(\xi_2-\xi_1)^2}\right)\left(\frac{2b^2(\xi_2-\xi_1)^2+(\xi_2-\xi_1)^2\xi_1\xi_2}{\xi_0\xi_1\xi_2}\right)\\
	\\
	=&\ \frac{(l-2)(-2\xi_1\xi_2)}{\xi_0\xi_1\xi_2}+\frac{8b^2-(\xi_2-\xi_1)^2}{\xi_0\xi_1\xi_2}+\frac{2b^2(l-2)+\xi_1\xi_2(l-2)}{\xi_0\xi_1\xi_2}\\
	\\
	=&\ \frac{2b^2(l+2)+(4-l)\xi_1\xi_2-\xi_1^2-\xi_2^2}{\sqrt{2}(\xi_1\xi_2)^{\frac{3}{2}}}\\
	\\
	=&\ \frac{2b^2(l+2)-(l+2)\xi_1\xi_2}{\sqrt{2}(\xi_1\xi_2)^{\frac{3}{2}}}\hspace{2.5cm} (\text{since } \xi_1+\xi_2=2\sqrt{2\xi_1\xi_2}\Longrightarrow\xi_1^2+\xi_2^2=6\xi_1\xi_2).
	\end{align*}
	Then
	\begin{center}
		$\displaystyle r_0-\frac{r_1+r_2}{2}=0\Longleftrightarrow\xi_1\xi_2=2b^2,$
	\end{center}
	by solving $\xi_1+\xi_2=2\sqrt{2\xi_1\xi_2}$ and $\xi_1\xi_2=2b^2$ for $\xi_1<\xi_2$, we obtain
	
	\begin{center}
		$\left\{\begin{array}{l}
		\xi_1=(-2+\sqrt{2})b\\
		\xi_2=(-2-\sqrt{2})b
		\end{array}\right.,$ $b<0,$ \hspace{1cm}	$\left\{\begin{array}{l}
		\xi_1=(2-\sqrt{2})b\\
		\xi_2=(2+\sqrt{2})b
		\end{array}\right.,$ $b>0,$
	\end{center}
	
	\noindent and $\xi_0=2|b|.$ Hence, an Einstein metric satisfying $\xi_0=\sqrt{2\xi_1\xi_2}$ must satisfy one of the conditions (F3), (F4), (F5) or (F6). Conversely, it is easy to verify that any invariant metric satisfying one of the conditions (F3), (F4), (F5),(F6) is in fact an Einstein metric.
\end{proof}
\noindent For $\Theta=\{\alpha_2,...,\alpha_{l-1}\}$ or $\{\alpha_2,...,\alpha_{l-2},\alpha_l\}$, the isotropy representation of $\mathbb{F}_\Theta$ also decomposes into three irreducible $K_\Theta-$invariant subspaces. Consider the automorphisms $\rho$ and $\eta$ of $\mathfrak{so}(l)\oplus\mathfrak{so}(l)$ given by

\begin{center}
	$\rho(w_{ij})=w_{l-j+1,l-i+1},$ $\rho(u_{ij})=u_{l-j+1,l-i+1}$, $1\leq j<i\leq l$,
	\
	
	$\eta(w_{ij})=w_{ij}$, $\eta(u_{ij})=u_{ij}$, $1\leq j<i\leq l-1,$
	\
	
	$\eta(w_{lj})=u_{lj}$, $\eta(u_{lj})=w_{lj}$, $1\leq j\leq l-1.$
\end{center}
We have that 
\begin{center}
	$\rho(\mathfrak{m}_{\{\alpha_{1},...,\alpha_{l-2}\}})=\mathfrak{m}_{\{\alpha_2,...,\alpha_{l-1}\}},$ $\eta(\mathfrak{m}_{\{\alpha_2,...,\alpha_{l-1}\}})=\mathfrak{m}_{\{\alpha_2,...,\alpha_{l-2},\alpha_l\}},$
\end{center}
and $\rho$, $\eta,$ take an invariant inner product to an invariant inner product. Consequently, the components of the Ricci tensor of invariant metrics for $\{\alpha_2,...,\alpha_{l-1}\}$ or $\{\alpha_2,...,\alpha_{l-2},\alpha_l\}$ are the same as in the case of $\{\alpha_1,...,\alpha_{l-2}\}$. Therefore, Einstein invariant metrics on $\mathbb{F}_{\{\alpha_2,...,\alpha_{l-1}\}}$ and $\mathbb{F}_{\{\alpha_2,...,\alpha_{l-2},\alpha_l\}}$ have the form $(\rho^{-1})^*g$ and $(\rho^{-1}\circ\eta^{-1})^*g$, respectively,  where $g$ is an Einstein invariant metric on $\mathbb{F}_{\{\alpha_1,...,\alpha_{l-2}\}}.$
\section{Equivalent metrics}
In this section, we shall decide which of the Einstein metrics found in the previous sections are equivalent in the sense of the following definition:\begin{dfn}\label{d4.3.1} Let $g_1,g_2$ be Riemannian metrics on a manifold $M.$ We say that $g_1$ and $g_2$ are equivalent if there exists an isometry $F:(M,g_1)\longrightarrow (M,g_2).$ 
\end{dfn}
As observed in \cite{Mu}, the Einstein constant corresponding to an Einstein Riemannian metric $g$ of volume 1 on a compact manifold $M$ is equal to $S/dim(M),$ where $S$ is the scalar curvature of $g$. Since any isometry preserves the scalar curvature, then two Einstein metrics of volume 1 with different Einstein constant cannot be equivalent. We denote by $v_0$ the volume of the flag with respect to the invariant inner product fixed in each case.

Let us consider the flag $(SO(l)\times SO(l+1))/SO(l),$ $l\geq 3$. By Propositions \ref{p4.3} and \ref{p4.10}, this manifold has two invariant Einstein metrics $A_i=(\mu_i,\gamma_i),\ i=1,2,$ where $\mu_1=\frac{\gamma_1}{2}$ and $\mu_2=\left(\frac{l}{2l-4}\right)\gamma_2$. When $\gamma_1=v_0^{-\frac{4}{l(l+1)}}2^{\frac{2}{l+1}}$ and $\gamma_2=v_0^{-\frac{4}{l(l+1)}}\left(\frac{2l-4}{l}\right)^{\frac{2}{l+1}}$ we have volume 1 and the corresponding Einstein constants are given by
\begin{center}
	$c_1=\frac{(l-1)v_0^{\frac{4}{l(l+1)}}}{2^{\frac{2}{l+1}-1}}$ and $c_2=\frac{(l-1)v_0^{\frac{4}{l(l+1)}}}{2^{\frac{2}{l+1}-1}}\left(\frac{l^{\frac{2}{l+1}-1}(l+2)}{(l-2)^{\frac{2}{l+1}-1}}\right).$
\end{center}
So, $c_1=c_2$ if and only if $\frac{l^{\frac{2}{l+1}-1}(l+2)}{(l-2)^{\frac{2}{l+1}-1}}=1$ or, equivalently, $\left(\frac{l}{l-2}\right)^{\frac{2}{l+1}}=\frac{l}{(l-2)(l+2)},$ which is not possible since $\frac{l}{(l-2)(l+2)}<1<\frac{l}{l-2}.$ Hence, $A_1$, $A_2$ are not equivalent.\\

\noindent For $SO(4)/S(O(2)\times O(1)\times O(1))$, denote by $A$ and $\tilde{A}$ the the invariant metrics satisfying (E2) and (E1) respectively (see Proposition \ref{p4.7}), their corresponding volumes are given by $2\sqrt{2}b^{\frac{5}{2}}$ and $\frac{9\mu_0^{\frac{5}{2}}}{16}$, and their corresponding Einstein constants are $\frac{1}{b}$ and $\frac{16}{9\mu_0}$. When $b=2^{-\frac{3}{5}}$ and $\mu_0=\left(\frac{4}{3}\right)^{\frac{4}{5}}$ we have volume 1, but $2^{\frac{3}{5}}\neq\left(\frac{4}{3}\right)^{\frac{6}{5}},$ so (E1) and (E2) cannot be equivalent. Now, if we consider the diffeomorphisms 
\begin{equation*}
	\psi_i:\mathbb{F}_{\{\alpha_1\}}\longrightarrow\mathbb{F}_{\{\alpha_1\}};\ \ \psi_i\left(kK_{\{\alpha_1\}}\right)=s_iks_i^TK_{\{\alpha_1\}},\ \ i=3,4,5;
\end{equation*}
where
\begin{equation*} s_3=\left(\begin{array}{cccc}
	1 & 0 & 0 & 0\\
	0 & -1 & 0 & 0\\
	0 & 0 & 0 & 1\\
	0 & 0 & 1 & 0\\
	\end{array}\right),\  s_4=\left(\begin{array}{cccc}
	0 & 1 & 0 & 0\\
	1 & 0 & 0 & 0\\
	0 & 0 & 1 & 0\\
	0 & 0 & 0 & 1\\
	\end{array}\right),\ \text{and}\ s_5=\left(\begin{array}{cccc}
	1 & 0 & 0 & 0\\
	0 & 1 & 0 & 0\\
	0 & 0 & 0 & 1\\
	0 & 0 & 1 & 0\\
	\end{array}\right).
\end{equation*}
Then it is easy to verify that the invariant metric $\psi_i^*A$ satisfies (Ei). This shows that (E2), (E3), (E4) and (E5) are equivalent. We can use a similar argument to show that metrics (F1), (F2) are equivalent as well as metrics (F3), (F4), (F5), (F6) and that (F1) is not equivalent to (F3) (see Proposition \ref{p4.13}).

\end{document}